\documentclass[12pt,leqno]{amsart}
\usepackage{mathrsfs,dsfont}
\usepackage{amsmath,amstext,amsthm,amssymb}

\newtheorem{Theorem}{Theorem}[section]
\newtheorem{Proposition}[Theorem]{Proposition}
\newtheorem{Lemma}[Theorem]{Lemma}
\newtheorem{Corollary}[Theorem]{Corollary}
\newtheorem{Definition}[Theorem]{Definition}

\newtheorem{Remark}[Theorem]{Remark}

\newcommand{\norm}[1]{\lVert #1\rVert}
\newcommand{\db}{\overline\partial}  
\newcommand{\ov}{\overline} 
\newcommand{\wi}{\widetilde}
\DeclareMathOperator{\ric}{Ric}
\DeclareMathOperator{\codim}{codim}
\DeclareMathOperator{\Dom}{Dom}
\newcommand{\cali}[1]{\mathscr{#1}}
\newcommand{\cO}{\cali{O}} \newcommand{\cI}{\cali{I}}

\begin{document}

\title{Equidistribution results for singular metrics on line bundles} 
\author{Dan Coman \and George Marinescu}
\thanks{D. Coman is supported by the NSF Grant DMS-0900934}
\thanks{G. Marinescu is partially supported by SFB TR 12}
\date{August 25, 2011}
\address{D. Coman: dcoman@@syr.edu, Department of Mathematics, Syracuse
University, Syracuse, NY 13244-1150, USA}
\address{G. Marinescu: gmarines@@math.uni-koeln.de, Universit\"at zu K\"oln, Mathematisches Institut, Weyertal 86-90, 50931 K\"oln, GERMANY}

\pagestyle{myheadings} 

\begin{abstract}
\noindent Let $(L,h)$ be a holomorphic line bundle with a positively curved singular Hermitian metric over  a complex manifold $X$. One can define naturally the sequence of Fubini-Study currents $\gamma_p$ associated to the space of $L^2$-holomorphic sections of $L^{\otimes p}$. Assuming that the singular set of the metric is contained in a compact analytic subset $\Sigma$ of $X$ and that the logarithm of the Bergman kernel function of $L^{\otimes p}\mid_{_{X\setminus\Sigma}}$ grows like $o(p)$ as $p\to\infty$, we prove the following:

\par 1) the currents $\gamma_p^k$ converge weakly on the whole $X$ to $c_1(L,h)^k$, where $c_1(L,h)$ is the curvature current of $h$. 

\par 2) the expectations of the common zeros of a random $k$-tuple of $L^2$-holomorphic sections converge weakly in the sense of currents to $c_1(L,h)^k$. \\
Here $k$ is so that $\codim\Sigma\geq k$. Our weak asymptotic condition on the Bergman kernel function is known to hold in many cases, as it is a consequence of its asymptotic expansion. We also prove it here in a quite general setting. We then show that many important geometric situations (singular metrics on big line bundles, K\"ahler-Einstein metrics on Zariski-open sets, artihmetic quotients) fit into our framework. 
\end{abstract}

\maketitle
\tableofcontents

\section{Introduction}
\par Let $X$ be a compact complex manifold of dimension $n$, $L\longrightarrow X$ be a positive holomorphic line bundle, and $h$ be a smooth Hermitian metric on $L$ whose curvature $c_1(L,h)$ is a positive (1,1) form on $X$. Let $\Phi_p:X\longrightarrow{\mathbb P}^{d_p-1}$ be the Kodaira map defined by an orthonormal basis of $H^0(X,L^p)$ with respect to the inner product given by the metric induced by $h$ on $L^p:=L^{\otimes p}$ and a fixed volume form on $X$, where $d_p=\dim H^0(X,L^p)$. The pull-back $\Phi_p^\star(\omega_{FS})$ of the Fubini-Study form $\omega_{FS}$ is a smooth (1,1) form for all $p$ sufficiently large, since $\Phi_p$ is an embedding by Kodaira's embedding theorem. A theorem of Tian \cite{Ti90} (with improvements by Ruan \cite{Ru98}) asserts that $\frac{1}{p}\,\Phi_p^\star(\omega_{FS})\to c_1(L,h)$ as $p\to\infty$, in the ${\mathcal C}^\infty$ topology on $X$.

\par Tian's theorem is a consequence of the asymptotic expansion of the Bergman kernel function associated to the inner product on $H^0(X,L^p)$ mentioned  above. In the context of positive line bundles this asymptotic expansion is proved in various forms in \cite{Ti90,Ca99,Z98,DLM04a,MM04,MM07,MM08,BBS08}. For line bundles endowed with arbitrary smooth Hermitian metrics the Bergman kernel function behavior and important consequences are studied in \cite{Be09} and \cite{BB10}.

\par In the case of holomorphic Hermitian line bundles over complete Hermitian manifolds the asymptotic expansion of the Bergman kernel function associated to the corresponding spaces of $L^2$-holomorphic sections was proved in \cite{MM08} (see also \cite{MM04,MM07}). In particular, a version of Tian's theorem was obtained for a big line bundle $L$ over a (compact) manifold $X$. Such a line bundle admits a singular Hermitian metric $h$, smooth outside a proper analytic subvariety $\Sigma\subset X$, and whose curvature current $c_1(L,h)$ is strictly positive. It is shown in \cite[Section 6.2]{MM07} that there exist a smooth positively curved Hermitian metric $h_\varepsilon$ on $L\mid_{_{X\setminus\Sigma}}$, which is a small perturbation of $h$, and a smooth positive (1,1) form $\Theta$ defining a generalized Poincar\'e metric on $X\setminus\Sigma$, so that the following hold. If $H^0_{(2)}(X\setminus\Sigma,L^p)$ is the space of $L^2$-holomorphic sections of $L^p\mid_{_{X\setminus\Sigma}}$ relative to the metrics $h_\varepsilon$ and $\Theta$ then $H^0_{(2)}(X\setminus\Sigma,L^p)\subset H^0(X,L^p)$, so a Kodaira map $\Phi_p:X\dashrightarrow{\mathbb P}^{d_p-1}$ can be defined by using an orthonormal basis of $H^0_{(2)}(X\setminus\Sigma,L^p)$. Let $\gamma_p=\Phi_p^\star(\omega_{FS})$ and $\omega=c_1(L\mid_{_{X\setminus\Sigma}},h_\varepsilon)$. Then $\frac{1}{p}\,\gamma_p\to\omega$ as $p\to\infty$, locally uniformly in the ${\mathcal C}^\infty$ topology on $X\setminus\Sigma$.

\par Since $\gamma_p$ are currents on $X$ it is natural to try and study the weak convergence of the sequence $\{\gamma_p/p\}$, and to ask whether a global version of Tian's theorem holds in this setting. We will show that this is indeed the case. 

\medskip

\par Let us work in the following more general setting:

\medskip

(A) $X$ is a complex manifold of dimension $n$ (not necessarily compact), $\Sigma$ is a compact analytic subvariety of $X$, and $\Omega$ is a smooth positive $(1,1)$ form on $X$.

\smallskip

(B) $(L,h)$ is a holomorphic line bundle on $X$ with a singular (semi)positively curved Hermitian metric $h$ which is continuous on $X\setminus\Sigma$.

\smallskip

(C) The volume form on $X\setminus\Sigma$ is $f\Omega^n$, where $f\in L^1_{loc}(X\setminus\Sigma,\Omega^n)$ verifies $f\geq c_x>0$ $\Omega^n$-a.e. in a neighborhood $U_x$ of each $x\in(X\setminus\Sigma)\cup\Sigma^{n-1}_{reg}$. Here $\Sigma^{n-1}_{reg}$ is the set of regular points $y$ where $\dim_y\Sigma=n-1$. 

\medskip

\par We denote the curvature current of $h$ by $\gamma=c_1(L,h)$ and consider the space $H^0_{(2)}(X\setminus\Sigma,L^p)$ of $L^2$-holomorphic sections of $L^p\mid_{_{X\setminus\Sigma}}$ relative to the metric $h_p$ induced by $h$ and the volume form  $f\Omega^n$ on $X\setminus\Sigma$, endowed with the inner product
$$(S,S')_p=\int_{X\setminus\Sigma}\langle S,S'\rangle_{h_p}\,f\Omega^n\,,\;\;{\rm where}\;\langle S,S'\rangle_{h_p}=h_p(S,S'),\;S,S'\in H^0_{(2)}(X\setminus\Sigma,L^p).$$ 
We let $\|S\|_p^2=(S,S)_p$. Since $H^0_{(2)}(X\setminus\Sigma,L^p)$ is separable, let $\{S^p_j\}_{j\geq1}$ be an orthonormal basis and denote by $P_p$ the Bergman kernel function defined by 
\begin{equation}\label{e:Bergfcn}
P_p(x)=\sum_{j=1}^\infty|S^p_j(x)|_{h_p}^2,\;\;|S^p_j(x)|_{h_p}^2:=\langle S_j^p(x),S_j^p(x)\rangle_{h_p},\;x\in X\setminus\Sigma.
\end{equation}
Note that this definition is independent of the choice of basis, and the function $P_p$ is continuous on $X\setminus\Sigma$ (see Section \ref{S:mtpf}). 

\par Next we define the Fubini-Study currents $\gamma_p$ on $X\setminus\Sigma$ by 
\begin{equation}\label{e:gammap}
\gamma_p\mid_{_U}=\frac{1}{2}\,dd^c\log\left(\sum_{j=1}^\infty|s_j^p|^2\right),\;U\subset X\setminus\Sigma\;\,{\rm open}\,, 
\end{equation}
where $d^c=\frac{1}{2\pi i}(\partial-\overline\partial)$, $S_j^p=s_j^pe^{\otimes p}$, and $e$ is a local holomorphic frame for $L$ on $U$. 

\par One of our main results is the following:

\begin{Theorem}\label{T:mt} If $X,\Sigma,(L,h),f,\Omega$ verify assumptions (A)-(C) then $H^0_{(2)}(X\setminus\Sigma,L^p)\subset H^0(X,L^p)$ and $\gamma_p$ extends to a positive closed current on $X$ defined locally by formula \eqref{e:gammap} and which is independent of the choice of basis $\{S_j^p\}_{j\geq1}$. Assume further that 
\begin{equation}\label{e:mainhyp}
\lim_{p\to\infty}\frac{1}{p}\,\log P_p(x)=0, \text{ locally uniformly on } X\setminus\Sigma.
\end{equation}
Then $\frac{1}{p}\,\gamma_p\to\gamma$ weakly on $X$. If, in addition, $\dim\Sigma\leq n-k$ for some $2\leq k\leq n$, then the currents $\gamma^k$ and $\gamma_p^k$ are well defined on $X$, respectively on each relatively compact neighborhood of $\Sigma$, for all $p$ sufficiently large. Moreover, $\frac{1}{p^k}\,\gamma_p^k\to\gamma^k$ weakly on $X$.
\end{Theorem}

\par This theorem is proved in Section \ref{S:mtpf}. The proof relies on a local property of the complex Monge-Amp\`ere operator which is of independent interest (see Theorem \ref{T:locMA}). Some background material about singular Hermitian metrics and pluripotential theory needed in the paper is recalled in Section \ref{S:prelim}.

\smallskip

\par We examine in Section \ref{S:ex} a series of important situations where condition \eqref{e:mainhyp} of Theorem \ref{T:mt} holds, as it is an immediate consequence of deep results regarding the asymptotic expansion of the Bergman kernel function $P_p(x)\sim b_0(x)p^n+b_1(x)p^{n-1}+\dots$. 
Especially, Theorem \ref{T:mt} yields equidistribution results for singular metrics on big line bundles (Sections \ref{SS:heps}, \ref{SS:sm}), on Zariski-open sets of bounded negative Ricci curvature (Section \ref{SS:qp}), on toroidal compactifications of arithmetic quotients (Section \ref{SS:aq}), and finally on $1$-convex manifolds (Sections \ref{SS:1conv}, \ref{SS:spscd}).

\par The point of view adopted in Theorem \ref{T:mt} is that once some information is known on the asymptotic behavior of $P_p$ on the set where the metric is continuous, then the global weak convergence on $X$ of the currents $\gamma_p/p$ and their powers follows. Hypothesis \eqref{e:mainhyp} is obviously a much weaker condition than the asymptotic expansion of $P_p$ mentioned above. Indeed, in Section \ref{S:Bkf} we give a simple proof that \eqref{e:mainhyp} holds in the case of line bundles over compact K\"ahler manifolds endowed with metrics that are assumed to be only continuous outside of $\Sigma$ (see Theorems \ref{T:Bkf} and \ref{T:mt2}). In this case the asymptotic expansion of $P_p$ is not known. 

\par We also prove in Theorem \ref{T:BkfK} that Tian's theorem \cite{Ti90} holds for any singular metric with strictly positive curvature. Namely, let $(X,\Omega)$ be a compact K\"ahler manifold and $(L,h)$ be a holomorphic line bundle on $X$ with a singular metric $h$ so that $c_1(L,h)$ is a strictly positive current. If $\gamma_p$ are the Fubini-Study currents defined by \eqref{e:gammap} for the spaces of $L^2$-holomorphic sections of $L^p$ relative to the metric induced by $h$ and the volume form  $\Omega^n$, then $\frac{1}{p}\,\gamma_p\to c_1(L,h)$ in the weak sense of currents on $X$. The proofs of Theorems \ref{T:BkfK} and \ref{T:Bkf} rely on techniques developed by Demailly \cite{D92,D93b}.

\medskip

\par In a series of papers including \cite{ShZ99,ShZ04,ShZ08}, Shiffman and Zelditch describe the asymptotic distribution of zeros of random holomorphic sections of a positive line bundle over a projective manifold endowed with a smooth positively curved metric. They also study the distribution of zeros of quantum ergodic eigenfunctions. To prove these results they develop interesting new techniques, based in part on methods in complex dynamics from \cite{FS95b}.

\par Later, using different methods, Dinh and Sibony \cite{DS06b} obtain sharper estimates for the speed of convergence in the asymptotic distribution of zeros of random holomorphic sections. In \cite{DMS} these results are generalized to the case of complete Hermitian manifolds. The problem of the distribution of zeros of random sections of line bundles appears in other contexts as well. For example, the case of canonical line bundles over towers of covers is studied in \cite{To01}.

\par We show here how some of the important results of Shiffman and Zelditch can be obtained in our setting from Theorem \ref{T:mt}, assuming in addition that $X$ is compact. More precisely, following the framework in \cite{ShZ99}, we let $\lambda_p$ be the normalized surface measure on the unit sphere ${\mathcal S}^p$ of $H^0_{(2)}(X\setminus\Sigma,L^p)$, defined in the natural way by using a fixed orthonormal basis (see Section \ref{S:SZ}). We denote by $\lambda_p^k$ the product measure on $({\mathcal S}^p)^k$, and by $[S=0]$ the current of integration (with multiplicities) over the analytic hypersurface $\{S=0\}$ determined by a nontrivial section $S\in H^0(X,L^p)$. We prove in Section \ref{S:SZ} the following generalization of some results of Shiffman and Zelditch \cite{ShZ99,ShZ08} to our situation:

\begin{Theorem}\label{T:SZ} In the setting of Theorem \ref {T:mt}, assume that $X$ is compact, $\dim\Sigma\leq n-k$ for some $1\leq k\leq n$, and that \eqref{e:mainhyp} holds. Then, for all $p$ sufficiently large:

\par (i) $[\sigma=0]:=[\sigma_1=0]\wedge\ldots\wedge[\sigma_k=0]$ is a well defined positive closed current of bidegree (k,k) on $X$, for $\lambda_p^k$-a.e. $\sigma=(\sigma_1,\dots,\sigma_k)\in({\mathcal S}^p)^k$. 

\par (ii) The expectation $E_p^k[\sigma=0]$ of the current-valued random variable $\sigma\to[\sigma=0]$, given by
$$\langle E_p^k[\sigma=0],\varphi\rangle=\int_{({\mathcal S}^p)^k}\langle[\sigma=0],\varphi\rangle\,d\lambda_p^k,$$
where $\varphi$ is a test form on $X$, is a well defined current and $E_p^k[\sigma=0]=\gamma_p^k$. 

\par (iii) We have $\frac{1}{p^k}\,E_p^k[\sigma=0]\to\gamma^k$ as $p\to\infty$, weakly in the sense of currents on $X$. 
\end{Theorem}

\par In particular, this theorem together with \cite[Lemma 3.3]{ShZ99} yields an equidistribution result for the zeros of a random sequence of sections $\{\sigma_p\}_{p\geq1}\in\prod_{p=1}^\infty{\mathcal S}^p$, i.e. $\frac{1}{p}\,[\sigma_p=0]\to\gamma$ as $p\to\infty$, in the weak sense of currents on $X$ (see Theorem \ref{T:rs}).

\medskip

\par {\em Acknowledgement.} Dan Coman is grateful to the Alexander von Humboldt Foundation for their support and to the Mathematics Institute at the University of Cologne for their hospitality.

\section{Preliminaries}\label{S:prelim}

\par We recall here a few of the notions that we will need. We start with the notion of singular Hermitian metric in Section \ref{SS:singm} and some necessary notions about desingularization in Section \ref{SS:desing}. In Section \ref{SS:metrics} we introduce the generalized Poincar\'e metric on a manifold and an associated metric on a line bundle with strictly positive curvature current. In Section \ref{SS:ppt} we recall a few facts regarding the definition of complex Monge-Amp\`ere operators.

\subsection{Singular Hermitian metrics on line bundles}\label{SS:singm}

\par Let $L\longrightarrow X$ be a holomorphic line bundle over a complex manifold $X$ and fix an open cover $X=\bigcup U_\alpha$ for which there exist local holomorphic frames $e_\alpha:U_\alpha\longrightarrow L$. The transition functions $g_{\alpha\beta}=e_\beta/e_\alpha\in{\mathcal O}^\star_X(U_\alpha\cap U_\beta)$ determine the \v{C}ech 1-cocycle $\{(U_\alpha,g_{\alpha\beta})\}$. 

\par Let $h$ be a smooth Hermitian metric on $L$. If $|e_\alpha(x)|_h^2=h(e_\alpha(x),e_\alpha(x))$ for $x\in U_\alpha$, we recall that the curvature form $c_1(L,h)$ of $h$ is defined by
$$c_1(L,h)\mid_{_{U_\alpha}}=-dd^c\log|e|_h=\frac{i}{2\pi}\,R^L,$$ 
where $R^L$ is the curvature of the holomorphic Hermitian connection $\nabla^L$ on $(L,h)$.

\par If $h$ is a singular Hermitian metric on $L$ then (see \cite{D90}, also \cite[p.\ 97]{MM07}) $h(e_\alpha,e_\alpha)=e^{-2\varphi_\alpha}$, where the functions $\varphi_\alpha\in L^1_{loc}(U_\alpha)$ are called the local weights of the metric $h$. One has $\varphi_\alpha=\varphi_\beta+\log|g_{\alpha\beta}|$ on $U_\alpha\cap U_\beta$, and the curvature of $h$, 
$$c_1(L,h)\mid_{_{U_\alpha}}=dd^c\varphi_\alpha,$$
is a well defined closed (1,1) current on $X$. We say that the metric $h$ is (semi)positively curved if $c_1(L,h)$ is a positive current. Equivalently, the weights $\varphi_\alpha$ can be chosen to be plurisubharmonic (psh) functions. 

\smallskip 

\par Let $L'\longrightarrow X$ be a holomorphic line bundle isomorphic to $L$. A metric $h^L$ on $L$ induces a metric $h^{L'}$ on $L'$ with curvature current $c_1(L,h^L)=c_1(L',h^{L'})$.

\par Suppose now that $M$ is a complex manifold and $f:M\longrightarrow X$ is a locally biholomorphic map. A metric $h^L$ on $L$ induces a metric $f^\star h^L$ on $f^\star L$ whose curvature current is   
$c_1(f^\star L,f^\star h^L)=f^\star\left(c_1(L,h^L)\right)$.

\subsection{Desingularization}\label{SS:desing}

We recall here Hironaka's embedded resolution of singularities theorem (see e.g. \cite{BM97}, \cite[Theorem 2.1.13]{MM07}). Let $X$ be a complex manifold and $\Sigma\subset X$ be a compact analytic subvariety of $X$. Then there exists a finite sequence of blow up maps $\sigma_{j+1}:X_{j+1}\longrightarrow X_j$ with smooth centers $Y_j$,  
$$\begin{array}{ll}X_m\stackrel{\sigma_m}{\longrightarrow}\\\Sigma_m\\E_m\end{array}
\begin{array}{ll}X_{m-1}\longrightarrow\ldots\\\Sigma_{m-1}\\E_{m-1}\end{array}
\begin{array}{ll}\longrightarrow\\ \\ \\ \end{array}
\begin{array}{ll}X_{j+1}\stackrel{\sigma_{j+1}}{\longrightarrow}\\\Sigma_{j+1}\\E_{j+1}\end{array}
\begin{array}{ll}X_j\longrightarrow\ldots\\\Sigma_j\\E_j\end{array}
\begin{array}{ll}\longrightarrow\\ \\ \\ \end{array}
\begin{array}{ll}X_1\stackrel{\sigma_1}{\longrightarrow}\\\Sigma_1\\E_1\end{array}
\begin{array}{ll}X_0=X\\\Sigma_0=\Sigma\\E_0=\emptyset\end{array},$$
such that: 

$(i)$ \ \ \ $Y_j$ is a compact submanifold of $X_j$ with $\dim Y_j\leq\dim X-2$ and $Y_j\subset\Sigma_j$,

$(ii)$ \ \ $\Sigma_{j+1}=\Sigma_j'$ is the strict transform of $\Sigma_j$ by $\sigma_{j+1}$,

$(iii)$ \ $E_{j+1}=E_j'\cup\sigma_{j+1}^{-1}(Y_j)$ is the set of exceptional hypersurfaces in $X_{j+1}$,

$(iv)$ \ \ $\Sigma_m$ is a smooth hypersurface and $\Sigma_m\cup E_m$ is a divisor with normal crossings.

\medskip

\par Let $\tau_j=\sigma_1\circ\ldots\circ\sigma_j:X_j\longrightarrow X$. Since $\sigma_{j+1}:X_{j+1}\setminus E_{j+1}\longrightarrow X_j\setminus(E_j\cup Y_j)$ is a biholomorphism, it follows that 
$$\tau_m:X_m\setminus E_m\longrightarrow\tau_m(X_m\setminus E_m)=X\setminus Y$$
is a biholomorphism, where
$$Y=Y_0\cup\tau_1(Y_1)\cup\tau_2(Y_2)\cup\ldots\cup\tau_{m-1}(Y_{m-1}).$$
As $Y_j\subset\Sigma_j$ and $\sigma_j(\Sigma_j)\subset\Sigma_{j-1}$, we have $\tau_j(Y_j)\subset\Sigma$ for every $j=1,\dots,m-1$. Since $Y_j$ is compact $\tau_j:Y_j\longrightarrow X$ is proper, so $\tau_j(Y_j)$ is an analytic subvariety of $X$ of dimension $\leq\dim Y_j$. Hence $Y$ is an analytic subvariety of $X$, $Y\subset\Sigma$ and $\dim Y\leq\dim X-2$. 

\medskip  

\par In conclusion, setting $\widetilde X=X_m$, $E=E_m$, and $\pi=\tau_m:\widetilde X\longrightarrow X$, we have: 

\begin{Theorem}[Hironaka]\label{T:Hironaka}Let $X$ be a complex manifold and $\Sigma\subset X$ be a compact analytic subvariety of $X$. Then there exist a complex manifold $\widetilde X$, an analytic subvariety $Y\subset\Sigma$ with $\dim Y\leq\dim X-2$, and a surjective holomorphic map $\pi:\widetilde X\longrightarrow X$ with the following properties:

(i) $\pi:\widetilde X\setminus E\longrightarrow X\setminus Y$ is a biholomorphism, where $E=\pi^{-1}(Y)$.

(ii) the strict transform $\Sigma'=\overline{\pi^{-1}(\Sigma\setminus Y)}$ is smooth and $\pi^{-1}(\Sigma)=\Sigma'\cup E$ is a divisor with normal crossings. 
\end{Theorem}

\subsection{Special metrics}\label{SS:metrics}

\par Let $X$ be a complex manifold of dimension $n$. Assume that $L\longrightarrow X$ is a holomorphic line bundle with a singular Hermitian metric $h$ which is continuous outside a proper compact analytic subvariety $\Sigma\subset X$, and whose curvature $\gamma=c_1(L,h)$ is a {\em strictly positive} closed (1,1) current on $X$ (i.e. it dominates a smooth positive (1,1) form on $X$). We write 
$$\Sigma=Z_1\cup Z_2,$$
where $Z_1,Z_2$ are analytic subvarieties of $X$, $Z_1$ has pure dimension $n-1$, and $\dim Z_2\leq n-2$. Let $\pi:\widetilde X\longrightarrow X$ be a resolution of singularities of $\Sigma$ as in Theorem \ref{T:Hironaka}. Then $\pi:\widetilde X\setminus E\longrightarrow X\setminus Y$ is a biholomorphism, where $Y\subset\Sigma$ is an analytic subvariety with $\dim Y\leq n-2$, $E=\pi^{-1}(Y)$, $Z_2\subset Y$, $\Sigma'=Z_1'$ is smooth, and $\pi^{-1}(\Sigma)=Z_1'\cup E$ is a divisor with normal crossings.

\subsubsection{The metric $\Theta$}\label{SSS:Theta}

\par We recall here the construction and properties of the generalized Poincar\'e metric on $X\setminus\Sigma$ (cf. \cite[Lemma 6.2.1]{MM07}). Let $\widetilde\Omega$ be a smooth positive (1,1) form on $\widetilde X$. When $X$ is not compact we take $\widetilde\Omega$ so that the associated metric is complete on $\widetilde X$.

\par Let $\Sigma_1,\dots,\Sigma_N$ be the irreducible components of $\pi^{-1}(\Sigma)$, so $\Sigma_j$ is a smooth hypersurface in $\widetilde X$. Let $\sigma_j$ be holomorphic sections of the associated holomorphic line bundle $\cO_{\widetilde X}(\Sigma_j)$ vanishing to first order on $\Sigma_j$ and let $|\cdot|_j$ be a smooth Hermitian metric on $\cO_{\widetilde X}(\Sigma_j)$ so that $|\sigma_j|_j<1$. We define  
\begin{equation}\label{e:F}
\widetilde\Theta_{\delta}=\widetilde\Omega+\delta dd^cF,\text{ where }\delta>0,\;F=-\frac{1}{2}\sum_{j=1}^N\log(-\log|\sigma_j|_j).
\end{equation}

\par If $\delta$ is small enough, $\widetilde\Theta_{\delta}$ defines a complete Hermitian metric on $\widetilde X\setminus\pi^{-1}(\Sigma)$ and we have $\widetilde\Theta_{\delta}\geq\widetilde\Omega/2$ in the sense of currents on $\widetilde X$. Moreover, if $X$ is compact then so is $\widetilde X$  and we have that $\widetilde\Theta_{\delta}$ has finite volume (see \cite[Lemma 6.2.1]{MM07}). Fixing such a $\delta$, we define the Poincar\'e metric on $X\setminus\Sigma$ as the metric associated to the $(1,1)$ form 
$$\Theta=(\pi^{-1})^\star\widetilde\Theta_\delta.$$
This metric has the same properties on $X\setminus\Sigma$ as $\widetilde\Theta_\delta$ does on $\widetilde X\setminus\pi^{-1}(\Sigma)$. 

\par Let now $x\in\Sigma^{n-1}_{reg}$ and local coordinates $z_1,\dots,z_n$ be chosen so that $x=0$, $\Sigma=\{z_1=0\}$. Then $\Theta^n\sim(|z_1|\log|z_1|)^{-2}\,d\lambda$ near $x$, where $\lambda$ is the Lebesgue measure in coordinates (see  \cite[(6.2.11)]{MM07}). In particular, we have that $\Theta^n=f\Omega^n$, where the function $f$ verifies assumption (C) stated in the introduction.

\subsubsection{The metric $h_\varepsilon$}\label{SSS:heps}

\par It is necessary to perturb the original metric $h$ of $L$ in order to obtain a metric on $L\mid_{_{X\setminus\Sigma}}$ whose curvature current dominates a small multiple of $\Theta$. By \cite[Lemma 6.2.2]{MM07} there exists a holomorphic line bundle $\widetilde L\longrightarrow\widetilde X$ which has a  singular Hermitian metric $h^{\widetilde L}$, continuous on $\widetilde X\setminus\pi^{-1}(\Sigma)$, and such that $\widetilde L\mid_{_{\widetilde X\setminus E}}$ is isomorphic to $\pi^\star\left(L^k\mid_{_{X\setminus Y}}\right)$, for some $k\in{\mathbb N}$. Moreover, $c_1(\widetilde L,h^{\widetilde L})=k\,\pi^\star\gamma+\theta$ is a strictly positive current on $\widetilde X$, where $\theta$ is a smooth real closed $(1,1)$ form supported in a neighborhood of $E$ and strictly positive along $E$. 

\par Since $\widetilde L\mid_{_{\widetilde X\setminus E}}\cong\pi^\star\left(L^k\mid_{_{X\setminus Y}}\right)$ the metric $h^{\widetilde L}$ induces a singular Hermitian metric $h^{L'}$ on $L'=\pi^\star\left(L\mid_{_{X\setminus Y}}\right)$ with curvature current $\gamma'=\pi^\star\gamma+\theta'$, where $\theta'=\theta/k$. For $\varepsilon>0$,
$$h_\varepsilon^{L'}=h^{L'}\prod_{j=1}^N(-\log|\sigma_j|_j)^\varepsilon$$
is a singular Hermitian metric on $L'$ with curvature current 
$$\gamma'_\varepsilon=\gamma'+\varepsilon dd^cF=\pi^\star\gamma+\theta'+\varepsilon dd^cF,$$
where $F$ is given in \eqref{e:F}. Since $\gamma'$ is a strictly positive current it follows that $\gamma'_\varepsilon$ is a strictly positive current on $\widetilde X$, provided that $\varepsilon$ is sufficiently small (cf. \cite[Lemma 6.2.1]{MM07}). We fix such an $\varepsilon$ and note that, as $\pi:\widetilde X\setminus E\longrightarrow X\setminus Y$ is a biholomorphism, the metric $h_\varepsilon^{L'}$ on $L'$ induces a singular metric $h_\varepsilon$ on $L\mid_{_{X\setminus Y}}$ which is continuous on $X\setminus\Sigma$. When $X$ is compact the curvature current of $h_\varepsilon$ dominates a small multiple of $\Theta$ on $X\setminus\Sigma$.

\subsection{Wedge products of singular currents}\label{SS:ppt}

\par We recall here a few facts that we need regarding the definition of complex Monge-Amp\`ere operators. Let $T$ be a positive closed current of bidimension $(l,l)$, $l>0$, on an open set $U$ in ${\mathbb C}^n$. The coefficients of $T$ are complex Radon measures and their total variations are dominated, up to multiplicative constants, by the trace measure of $T$, $|T|=T\wedge\Omega^l$, where $\Omega$ is any fixed smooth positive $(1,1)$ form on $U$. If $u$ is a psh function on $U$ so that $u\in L^1_{loc}(U,|T|)$ we say that the wedge product $dd^cu\wedge T$ is well defined. This is the positive closed current of bidimension $(l-1,l-1)$ defined by $dd^cu\wedge T=dd^c(uT)$. 

\par If $u_1,\dots,u_q$ are psh functions on $U$ we say that $dd^cu_1\wedge\ldots\wedge dd^cu_q$ is well defined if one can define inductively as above all intermediate currents 
$$dd^cu_{j_1}\wedge\ldots\wedge dd^cu_{j_l}=dd^c(u_{j_1}dd^cu_{j_2}\wedge\ldots\wedge dd^cu_{j_l}),\;1\leq j_1<\ldots<j_l\leq q.$$ 
The wedge product is well defined for locally bounded psh functions \cite{BT76,BT82}, for psh functions that are locally bounded outside a compact subset of a pseudoconvex open set $U$, or when the mutual intersection of their unbounded loci is small in a certain sense \cite{Si85,D93, FS95}. We recall here one such situation \cite[Corollary 2.11]{D93}: if $u_1,\dots,u_q$ are psh functions on $U$ so that $u_j$ is locally bounded outside an analytic subset $A_j$ of $U$ and ${\rm codim}\,(A_{j_1}\cap\ldots\cap A_{j_l})\geq l$ for each $l$, $1\leq j_1<\ldots<j_l\leq q$, then $dd^cu_1\wedge\ldots\wedge dd^cu_q$ is well defined. 
We also note that the natural domain of definition of the Monge-Amp\`ere operator $u\to(dd^cu)^n$ is completely described in \cite{Bl04,Bl06,Ce04}. 

\par If $T$ is a positive closed current of bidegree $(1,1)$ on a complex manifold $X$ then locally $T=dd^cu$ for a psh function $u$. Hence defining $T_1\wedge\ldots\wedge T_q$ for such currents $T_j$ amounts to verifying locally one of the conditions mentioned above for their psh potentials $u_j$. We conclude this brief overview by noting that when $X$ is compact the class of currents for which the wedge product can be globally defined so that it has good continuity properties is larger than the one for which it is well defined by local considerations as above (see e.g. \cite{GZ05,GZ07,CGZ08}).

\section{Proof of Theorem \ref{T:mt}}\label{S:mtpf}

\par In this section we give the proof of Theorem \ref{T:mt}. We start with a rather elementary property of the Bergman kernel function $P_p$ in Lemma \ref{L:Bergfcn} and show in Lemma \ref{L:gammap} that $\log P_p$ is, locally on $X$, the difference of two psh functions. Moreover, the Fubini-Study currents $\gamma_p$ are well defined, and if the codimension of $\Sigma$ is bigger than $k\geq2$, then also the wedge products $\gamma_p^k$ are well defined (Lemma \ref{L:gammapk}). We continue with the crucial Theorem \ref{T:locMA} about the local continuity properties of the Monge-Amp\`ere operator. This result is of independent interest. With these preparations we can then prove Theorem \ref{T:mt}.
 
\par For the convenience of the reader, we include a proof of the following properties of the function $P_p$ in our setting.

\begin{Lemma}\label{L:Bergfcn} If $P_p$ is the Bergman kernel function defined in \eqref{e:Bergfcn} then the definition is independent of the basis $\{S^p_j\}_{j\geq1}$ and the function $P_p$ is continuous on $X\setminus\Sigma$.
\end{Lemma} 

\begin{proof} By the Riesz-Fischer theorem we have that $S\in H^0_{(2)}(X\setminus\Sigma,L^p)$ if and only if there exists a sequence $a=\{a_j\}\in l^2$ so that $S=S_a$, where $S_a=\sum_{j=1}^\infty a_jS^p_j$ and  $\|S_a\|_p=\|a\|_2$. 

\smallskip

\par Fix $x\in X\setminus\Sigma$ and a neighborhood $U_\alpha\subset\subset X\setminus\Sigma$ of $x$ with a holomorphic frame $e_\alpha$ of $L$ over $U_\alpha$ and so that $f\geq c>0$ on $U_\alpha$. We write $S_a=s_ae_\alpha^{\otimes p}$, $S^p_j=s^p_je_\alpha^{\otimes p}$, and we let $\psi_\alpha$ be a continuous psh weight of $h$ on $U_\alpha$. It follows that $s_a=\sum_{j=1}^\infty a_js^p_j$ and the series converges locally uniformly on $U_\alpha$. As this holds for every sequence $a\in l^2$ we have that $\{s^p_j(z)\}\in l^2$ for all $z\in U_\alpha$. 

\par We fix compacts $K_i$ so that $x\in{\rm int}\,K_1$, $K_1\subset\subset K_2\subset\subset K_3\subset U_\alpha$. For $z\in K_2$ consider the sections $S_z=\sum_{j=1}^\infty\overline{s_j^p(z)}S^p_j\in H^0_{(2)}(X\setminus\Sigma,L^p)$, and write $S_z=s_ze_\alpha^{\otimes p}$. Then 
$$\left(\sum_{j=1}^\infty|s_j^p(z)|^2\right)^2=|s_z(z)|^2\leq C_1\int_{K_3}|s_z|^2e^{-2p\psi_\alpha}\,f\Omega^n\leq C_1\|S_z\|^2_p=C_1\sum_{j=1}^\infty|s_j^p(z)|^2,$$
for some constant $C_1$. This implies that 
$$\sum_{j=1}^\infty|s_j^p(z)|^2\leq C_1,\;\forall\,z\in K_2.$$
We have 
$$|s_j^p(y)|^2\leq C_2\int_{K_2}|s_j^p|^2\Omega^n,\;\forall\,y\in K_1,$$ 
where $C_2$ is a constant. Therefore 
$$\sum_{j=1}^\infty\max_{K_1}|s_j^p|^2\leq C_2\int_{K_2}\left(\sum_{j=1}^\infty|s_j^p|^2\right)\Omega^n\leq C_1C_2\int_{K_2}\Omega^n,$$
so the series $\sum_{j=1}^\infty|s_j^p|^2$ converges uniformly on $K_1$. This shows that the function $P_p=\sum_{j=1}^\infty|s_j^p|^2e^{-2p\psi_\alpha}$ is continuous near $x$. 

\smallskip

\par To see that $P_p$ does not depend on the choice of basis, observe that 
$$P_p(x)=\max\{|S(x)|^2_{h_p}:\,S\in H^0_{(2)}(X\setminus\Sigma,L^p),\;\|S\|_p=1\}.$$ 
Indeed, using the above notations we have for $a\in l^2$ with $\|a\|_2=1$, 
$$|S_a(x)|^2_{h_p}=\left|\sum_{j=1}^\infty a_js_j^p(x)\right|^2e^{-2p\psi_\alpha(x)}\leq\left(\sum_{j=1}^\infty|s^p_j(x)|^2\right)e^{-2p\psi_\alpha(x)}=P_p(x).$$
Moreover, if 
$$a=\left\{c^{-1}\,\overline{s_j^p(x)}\right\}_{j\geq1},\;c:=\left(\sum_{j=1}^\infty|s_j^p(x)|^2\right)^{1/2},$$
then $\|a\|_2=1$, $S_a(x)=ce_\alpha^{\otimes p}$, so $|S_a(x)|^2_{h_p}=P_p(x)$. 
\end{proof}

\par We start the proof of Theorem \ref{T:mt} with two lemmas. 

\begin{Lemma}\label{L:gammap} If $X,\,\Sigma,\,(L,h),\,f,\,\Omega$ are as in Theorem \ref{T:mt} then:

\par (i) $H^0_{(2)}(X\setminus\Sigma,L^p)\subset H^0(X,L^p)$. 

\par (ii) $\gamma_p$ extends to a positive closed current of bidegree (1,1) on $X$ defined locally by formula \eqref{e:gammap} and which is independent of the choice of basis $\{S_j^p\}$.

\par $(iii)$ $\log P_p\in L^1_{loc}(X,\Omega^n)$ and $dd^c\log P_p=2\gamma_p-2p\gamma$ as currents on $X$.
\end{Lemma}

\begin{proof} $(i)$ Let $x\in\Sigma^{n-1}_{reg}$ and let $e_\alpha$ be a holomorphic frame of $L$ on a neighborhood $U_\alpha$ of $x$. A section $S\in H^0_{(2)}(X\setminus\Sigma,L^p)$ can be written on $U_\alpha$ as $S=se_\alpha^{\otimes p}$ where $s$ is a holomorphic function on $U_\alpha\setminus\Sigma$. We may assume that $h$ has a psh weight $\psi_\alpha$ which is bounded above on $U_\alpha$ and that $f\geq c>0$ on $U_\alpha$ for some constant $c$. Then 
$$\int_{U_\alpha\setminus\Sigma}|s|^2\,\Omega^n\leq C\int_{U_\alpha\setminus\Sigma}|s|^2e^{-2p\psi_\alpha}\,f\Omega^n\leq C\|S\|^2_p<\infty.$$
By Skoda's lemma \cite[Lemma 2.3.22]{MM07}, this implies that $S$ extends holomorphically near $x$. 

\par Thus any section $S\in H^0_{(2)}(X\setminus\Sigma,L^p)$ extends holomorphically to a section of $L^p$ over $X\setminus Y$, where $Y=\Sigma\setminus\Sigma^{n-1}_{reg}$, and hence to a holomorphic section of $L^p$ since $Y$ is an analytic subvariety of $X$ of codimension $\geq2$. 

\smallskip

\par$(ii)$ Let $x\in\Sigma$, $U_\alpha$ be a neighborhood of $x$ as above, and set 
$$u_p:=\log\left(\sum_{j=1}^\infty|s_j^p|^2\right)\;{\rm on}\;U_\alpha.$$ 
By Lemma \ref{L:Bergfcn} it follows that the function $e^{u_p}=P_pe^{2p\psi_\alpha}$ is continuous on $U_\alpha\setminus\Sigma$, hence $u_p$ is psh since it satisfies the subaverage inequality. This implies that $\gamma_p$ is a positive closed current on $X\setminus\Sigma$. We may assume that there exists coordinates $(z_1,\dots,z_n)$ on $U_\alpha$ so that $x=0$ and $U_\alpha\cap\Sigma$ is contained in the cone $\{|z_n|\leq\max(|z_1|,\dots,|z_{n-1}|)\}$. Applying the maximum principle on complex lines parallel to the $z_n$ axis, we see that there exist a neighborhood $V\subset  U_\alpha$ of $x$ and a compact set $K\subset U_\alpha\setminus\Sigma$ so that $\sup_{z\in V}e^{u_p}\leq\sup_{z\in K}e^{u_p}$. 
This implies that the function $u_p$ is bounded above, hence psh, on $V$. So $\gamma_p$ extends to a positive closed current on $X$ defined locally by formula \eqref{e:gammap}. The current $\gamma_p$ does not depend on the choice of basis $\{S^p_j\}$ since the function $P_p$ is independent of the choice of basis.

\smallskip

\par$(iii)$ If $x\in\Sigma$ and $U_\alpha$ are as above, then by $(ii)$ $\log P_p=u_p-2p\psi_\alpha$, $\Omega^n$-a.e. on $U_\alpha$. So $\log P_p$ is locally the difference of two psh functions.
\end{proof}

\begin{Lemma}\label{L:gammapk} If $\dim \Sigma\leq n-k$ for some $2\leq k\leq n$ and hypothesis \eqref{e:mainhyp} holds then the currents $\gamma^k$ and $\gamma_p^k$ are well defined on $X$, respectively on each relatively compact neighborhood of $\Sigma$, for all $p$ sufficiently large.  
\end{Lemma}

\begin{proof} The current $\gamma^k$ is well defined by \cite[Corollary 2.11]{D93}, since $\dim \Sigma\leq n-k$. 

\par Let $A_p=\{x\in X:\,S^p_j(x)=0,\;\forall\,j\geq1\}$. Lemma \ref{L:gammap} shows that the current $\gamma_p$ has local psh potentials which are continuous away from $A_p$. By \cite[Corollary 2.11]{D93}, it suffices to show that given any relatively compact neighborhood $U$ of $\Sigma$ we have $\dim(A_p\cap U)\leq n-k$ for all $p$ sufficiently large. 

\par Assuming the contrary, there exist $m>n-k$ and a sequence $p_j\to\infty$ so that each analytic set $A_{p_j}\cap U$ has an irreducible component $Y_j$ of dimension $m$. It follows from \eqref{e:mainhyp} that, given any $\varepsilon$-neighborhood $V_\varepsilon$ of $\Sigma$, $Y_j\subset A_{p_j}\cap U\subset V_\varepsilon$ for all $j$ sufficiently large, hence $Y_j$ are compact. Let $|Y_j|=\int_{Y_j}\Omega^m$ and $T_j=[Y_j]/|Y_j|$, where $[Y_j]$ denotes the current of integration on $Y_j$. Since $T_j$ have unit mass, we may assume by passing to a subsequence that $T_j$ converges weakly to a positive closed current $T$ of bidimension $(m,m)$. But $T$ is supported by $\Sigma$, so $T=0$ by the support theorem as $\dim\Sigma\leq n-k<m$. On the other hand $\langle T,\Omega^m\rangle=\lim_{j\to\infty}\langle T_j,\Omega^m\rangle=1$, a contradiction. 
\end{proof}

\medskip

\par We will need the following local property of the complex Monge-Amp\`ere operator:

\medskip

\begin{Theorem}\label{T:locMA} Let $U$ be an open set in ${\mathbb C}^n$, $\Sigma$ be a proper analytic suvariety of $U$, and $v$ be a psh function on $U$ which is continuous on $U\setminus\Sigma$. Assume that $v_p$, $p\geq1$, are psh functions on $U$ so that $v_p\to v$ locally uniformly on $U\setminus\Sigma$. Then:

\par (i) The sequence $\{v_p\}$ is locally uniformly upper bounded in $U$.

\par (ii) Assume that $\dim\Sigma\leq n-k$ and the currents $(dd^cv_p)^k$ are well defined on $U$ for some $k\geq1$. Then $(dd^cv_p)^k\to(dd^cv)^k$ weakly in the sense of currents on $U$.
\end{Theorem}

\begin{proof} $(i)$ The sequence $\{v_p\}$ is clearly locally uniformly upper bounded in $U\setminus\Sigma$. If $x\in\Sigma$ we may assume that there exists coordinates $(z_1,\dots,z_n)$ on some neighborhood $V$ of $x=0$ so that $V\cap\Sigma$ is contained in the cone $\{|z_n|\leq\max(|z_1|,\dots,|z_{n-1}|)\}$. Applying the maximum principle on complex lines parallel to the $z_n$ axis, we see that there exist a neighborhood $V_1\subset V$ of $x$ and a compact set $K\subset V\setminus\Sigma$ so that $\sup_{V_1}v_p\leq\sup_Kv_p$. Hence $\{v_p\}$ is uniformly upper bounded on $V_1$.

\medskip

\par $(ii)$ Recall that the current $(dd^cv)^k$ is well defined on $U$ since $\dim\Sigma\leq n-k$ \cite[Corollary 2.11]{D93}). Since $v_p\to v$ locally uniformly on $U\setminus\Sigma$ and $v$ is continuous there we have that $(dd^cv_p)^k\to(dd^cv)^k$ weakly in the sense of currents on $U\setminus\Sigma$ (see e.g. \cite{BT76,BT82}, also \cite[Corollary 1.6]{D93}). We divide the proof into three steps. 

\smallskip

\par{\em Step 1.} We prove here assertion $(ii)$ when $k=n$. Then $\Sigma$ consists of isolated points of $U$. Let $x\in\Sigma$ and $\chi\geq0$ be a smooth function with compact support in $U$ so that $\chi=1$ near $x$ and ${\rm supp}\,\chi\cap\Sigma=\{x\}$. Then 
$$\int\chi\,(dd^cv_p)^n=\int v_p(dd^cv_p)^{n-1}\wedge dd^c\chi\to\int v(dd^cv)^{n-1}\wedge dd^c\chi=\int\chi\,(dd^cv)^n,$$
since $v_p\to v$ locally uniformly in a neighborhood of ${\rm supp}\,dd^c\chi$ and $v$ is continuous there.
This shows that the sequence of positive measures $(dd^cv_p)^n$ has locally bounded mass on $U$ and that if $\nu$ is any weak limit point of this sequence then $\nu(\{x\})=(dd^cv)^n(\{x\})$ for each $x\in\Sigma$. It follows that  $(dd^cv_p)^n\to(dd^cv)^n$ weakly in the sense of measures on $U$.

\smallskip

\par We assume in the sequel that $1\leq k\leq n-1$. 

\smallskip

\par{\em Step 2.} We show that the currents $(dd^cv_p)^k$ have locally uniformly bounded mass in $U$. Note that we only have to show this near points $x\in\Sigma$. The proof is quite standard in the case $k=1$ and when $k>1$ it follows from Oka's inequality for currents due to Forn\ae ss and Sibony \cite{FS95}. 

\par Consider first the case $k=1$. Fix $V\subset U$ a relatively compact neighborhood of $x$ and compacts $K_j\subset V$ so that $x\in{\rm int}\,K_1$, $K_1\subset{\rm int}\,K_2$, and $K_3\subset V\setminus\Sigma$ has positive Lebesgue measure. Subtracting a constant we may assume that $v_p,v<0$ on $V$. There exists a constant $C(K_1,K_2)$ so that $\|dd^cv_p\|_{K_1}\leq C(K_1,K_2)\int_{K_2}|v_p|$ for every $p$ (see e.g. \cite[Remark 1.4]{D93}). By \cite[Theorem 3.2.12]{Ho}, the family of psh functions $u$ on $V$ so that $u<0$ and $\int_{K_3}|u|=1$ is compact in $L^1_{loc}(V)$. Hence there exists a constant $C(K_2,K_3)$ so that $\int_{K_2}|v_p|\leq C(K_2,K_3)\int_{K_3}|v_p|$ for every $p$. We conclude that the currents $dd^cv_p$ have uniformly bounded mass on $K_1$.

\smallskip

\par Asume next that $2\leq k\leq n-1$. Let $x$ be a regular point of $\Sigma$ so that $\dim_x\Sigma=n-k$.  By a change of coordinates near $x$ we may assume that 
$$x=(1/2,\dots,1/2)\in\overline\Delta^n\subset U\,,\;\Sigma\cap\overline\Delta^n=\{z=(z_1,\ldots,z_n):\,z_1=\ldots=z_k=1/2\},$$ where $\Delta$ is the unit disc in ${\mathbb C}$. We may also assume that $v_p,v<0$ near $\overline\Delta^n$. Consider the $(k-1,n-k+1)$ Hartogs figure 
\begin{eqnarray*}
H&=&\{(z',z'')\in{\mathbb C}^{k-1}\times{\mathbb C}^{n-k+1}:\,\|z'\|\leq1,\,\|z''\|\leq1/4\}\cup\\ &&
\{(z',z'')\in{\mathbb C}^{k-1}\times{\mathbb C}^{n-k+1}:\,3/4\leq\|z'\|\leq1,\,\|z''\|\leq1\},
\end{eqnarray*}
where $\|z'\|=\max(|z_1|,\dots|z_{k-1}|)$. The current $T=v_p(dd^cv_p)^{k-1}$ is a negative current near $\overline\Delta^n$ of bidegree $(k-1,k-1)$ and $dd^cT=(dd^cv_p)^k\geq0$. By Oka's inequality applied to $T$ \cite[Theorem 2.4]{FS95},
$$\|v_p(dd^cv_p)^{k-1}\|_K+\|(dd^cv_p)^k\|_K\leq C\|v_p(dd^cv_p)^{k-1}\|_H$$
for some absolute constant $C$, where $K=\overline\Delta^n_{3/4}$ is the polydisc of radius $3/4$. Note that $x\in{\rm int}\,K$. Since $H\cap\Sigma=\emptyset$ we have $v_p(dd^cv_p)^{k-1}\to v(dd^cv)^{k-1}$ near $H$. It follows that $\|(dd^cv_p)^k\|_K$ are uniformly bounded. 

\par Therefore we showed that the currents $(dd^cv_p)^k$ have locally bounded mass on $U\setminus Y$, where $Y\subset\Sigma$ is an analytic set of codimension $\geq k+1$. Oka's inequality applied to the currents $(dd^cv_p)^k$ implies that they have locally uniformly bounded mass near each $y\in Y$ (see also \cite[Corollary 2.6]{FS95}). 

\smallskip

\par{\em Step 3.} We now prove that $(dd^cv_p)^k\to(dd^cv)^k$ weakly on $U$. Since the currents $(dd^cv_p)^k$ have locally uniformly bounded mass on $U$, it suffices to prove that any weak limit point $T$ of $(dd^cv_p)^k$ is equal to $(dd^cv)^k$. Let us write 
$$\Sigma=Y\cup\bigcup_{j\geq1}Y_j,$$ 
where $Y_j$ are the irreducible components of dimension $n-k$ and $\dim Y\leq n-k-1$. Recall that $T=(dd^cv)^k$ on $U\setminus\Sigma$. Hence by Federer's support theorem (\cite{Fed69}, see also \cite[Theorem 1.7]{Har77}), $T=(dd^cv)^k$ on $D=U\setminus\cup_{j\geq1}Y_j$, since $Y$ is an analytic subvariety of $D$ of dimension $\leq n-k-1$. 

\par By Siu's decomposition formula (\cite{Siu74}, see also \cite[Theorem 6.19]{D93}) we write 
$$T=R+\sum_{j\geq1}c_j[Y_j]\,,\;(dd^cv)^k=S+\sum_{j\geq1}d_j[Y_j]\,,$$
where $[Y_j]$ denotes the current of integration on $Y_j$, $c_j,d_j\geq0$, and $R,S$ are positive closed currents of bidegree $(k,k)$ on $U$ which do not charge any $Y_j$ (i.e. the trace measure of $R$ is 0 on $Y_j$). It follows by above that $R=S$. To conclude the proof we have to show that $c_j=d_j$ for each $j$. This will be done using slicing.

\par Without loss of generality, let $j=1$ and $x\in Y_1$ be a regular point of $\Sigma$. 
By a change of coordinates $z=(z',z'')$ near $x$ we may assume that $x=0\in\overline\Delta^n\subset U$ and $\Sigma\cap\Delta^n=Y_1\cap\Delta^n=\{z'=0\}$, where $z'=(z_1,\dots,z_k)$,  $z''=(z_{k+1},\dots,z_n)$. Since $v_p\to v$ locally uniformly on $U\setminus\Sigma$ and $v$ is continuous there, it follows that for each $z''\in\Delta^{n-k}$ the functions $v_{p,z''}(z')=v_p(z',z'')$, $v_{z''}(z')=v(z',z'')$, are locally bounded near the boundary of $\Delta^k$, so their Monge-Amp\`ere measures $(dd^cv_{p,z''})^k$, $(dd^cv_{z''})^k$ are well defined on $\Delta^k$ (see \cite[Corolary 2.3]{D93}). Arguing as in the proof of Step 1, it follows that $(dd^cv_{p,z''})^k\to(dd^cv_{z''})^k$ weakly on $\Delta^k$ as $p\to\infty$, for each $z''\in\Delta^{n-k}$. 

\par Let $\chi_1(z')\geq0$ (resp. $\chi_2(z'')\geq0$) be a smooth function with compact support in $\Delta^k$ (resp. $\Delta^{n-k}$) so that $\chi_1=1$ near $0\in{\mathbb C}^k$ (resp. $\chi_2=1$ near $0\in{\mathbb C}^{n-k}$). Let $\beta=i/2\sum_{j=k+1}^ndz_j\wedge d\overline z_j$ be the standard K\"ahler form in ${\mathbb C}^{n-k}$. One has the slicing formula (see e.g. \cite[formula (2.1)]{DS06})
$$\int_{\Delta^n}\chi_1(z')\chi_2(z'')(dd^cv_p)^k\wedge\beta^{n-k}=\int_{\Delta^{n-k}}\left(\int_{\Delta^k}\chi_1(z')(dd^cv_{p,z''})^k\right)\chi_2(z'')\beta^{n-k},$$
and similarly for $(dd^cv)^k$. Note that 
$$\int\chi_1\,(dd^cv_{p,z''})^k=\int v_{p,z''}(dd^cv_{p,z''})^{k-1}\wedge dd^c\chi_1.$$
Since $dd^c\chi_1$ is supported away from $\Sigma$, the Chern-Levine-Nirenberg estimates imply that this integral is locally uniformly bounded as a function of $z''$. Letting $p\to\infty$ we infer by above that 
$$\int_{\Delta^n}\chi_1(z')\chi_2(z'')T\wedge\beta^{n-k}=\int_{\Delta^n}\chi_1(z')\chi_2(z'')(dd^cv)^k\wedge\beta^{n-k}.$$

\par By Siu's decomposition formulas of $T$ and $(dd^cv)^k$, and since $R=S$, this implies that 
$$c_1\int_{\{z'=0\}}\chi_2(z'')\beta^{n-k}=d_1\int_{\{z'=0\}}\chi_2(z'')\beta^{n-k}.$$
As $\int_{\{z'=0\}}\chi_2(z'')\beta^{n-k}>0$ we see that $c_1=d_1$, and the proof is complete.
\end{proof}

\medskip

\par We finish now the proof of Theorem \ref{T:mt} by showing that $\frac{1}{p^k}\,\gamma_p^k\to\gamma^k$ weakly on $X$. Since this is local, we fix $x\in X$ and let $U_\alpha$ be a relatively compact neighborhood of $x$ such that there exists a holomorphic frame $e_\alpha$ of $L$ over $U_\alpha$. Let $\psi_\alpha$ be a psh weight of $h$ on $U_\alpha$ and let 
$$v_p=\frac{1}{2p}\log\left(\sum_{j=1}^\infty|s_j^p|^2\right),\;\text{ where }S_j^p=s_j^pe_\alpha^{\otimes p},\;s_j^p\in{\mathcal O}(U_\alpha).$$
By Lemma \ref{L:gammap} the function $v_p$ is psh on $U_\alpha$ and we have $\frac{1}{p}\,\gamma_p=dd^cv_p$, $\gamma=dd^c\psi_\alpha$. Moreover, Lemma \ref{L:gammapk} shows that the currents $(dd^cv_p)^k$ are well defined on $U_\alpha$ for all $p$ sufficiently large. Note that $\psi_\alpha$ is continuous on $U_\alpha\setminus\Sigma$. Since $\frac{1}{2p}\,\log P_p=v_p-\psi_\alpha$, 
hypothesis \eqref{e:mainhyp} implies that $v_p\to\psi_\alpha$ locally uniformly on $U_\alpha\setminus\Sigma$. It follows by Theorem \ref{T:locMA} that $\frac{1}{p^k}\,\gamma_p^k=(dd^cv_p)^k\to(dd^c\psi_\alpha)^k=\gamma^k$ weakly on $U_\alpha$. This concludes the proof of Theorem \ref{T:mt}.

\begin{Remark}\label{R:logPp} In the setting of Theorem \ref{T:mt}, assume that $\dim\Sigma\leq n-k$ and that \eqref{e:mainhyp} holds. The proof of Lemma \ref{L:gammapk} shows that all currents $\gamma_p^j\wedge\gamma^l$, $j+l\leq k$ are well defined positive closed currents on $X$. By Lemma \ref{L:gammap} $\log P_p\in L^1_{loc}(X,\Omega^n)$ and $dd^c\log P_p=2\gamma_p-2p\gamma$ is a current of order 0 on $X$. It follows that $(dd^c\log P_p)^j$, $j\leq k$, are currents of order 0 on $X$ which can be defined inductively by 
$$(dd^c\log P_p)^{j+1}=dd^c\left(\log P_p\,(dd^c\log P_p)^j\right),\;j<k,$$ 
since locally, $\log P_p$ is integrable with respect to the measure coefficients of $(dd^c\log P_p)^j$. Moreover, we have 
$$\left(\frac{1}{2p}\,dd^c\log P_p\right)^k=\left(\frac{1}{p}\,\gamma_p-\gamma\right)^k=\sum_{j=0}^k\binom{k}{j}\frac{(-1)^{k-j}}{p^j}\,\gamma_p^j\wedge\gamma^{k-j}.$$
A straightforward adaptation of the proof of Theorems \ref{T:mt} and \ref{T:locMA} shows that 
$$p^{-j}\gamma_p^j\wedge\gamma^{k-j}\to\gamma^k\,,\;{\rm as}\;p\to\infty,$$ 
weakly on $X$. Hence $p^{-j}(dd^c\log P_p)^j\to0$ as $p\to\infty$ in the weak sense of currents of order 0 on $X$, for all $1\leq j\leq k$.
\end{Remark}

\begin{Remark}\label{R:mtalt} Observe that the hypothesis $f\geq c_x>0$ $\Omega^n$-a.e. in a neighborhood $U_x$ of each $x\in\Sigma^{n-1}_{reg}$ was only needed in the proof of Lemma \ref{L:gammap} (i), i.e. to show that $H^0_{(2)}(X\setminus\Sigma,L^p)\subset H^0(X,L^p)$. Therefore, Theorem \ref{T:mt} also holds provided that $X,\Sigma,(L,h),f,\Omega$ verify assumptions (A), (B), (C') and (D), where:

\medskip

\par (C') The volume form on $X\setminus\Sigma$ is $f\Omega^n$, where $f\in L^1_{loc}(X\setminus\Sigma,\Omega^n)$ verifies $f\geq c_x>0$ $\Omega^n$-a.e. in a neighborhood $U_x$ of each $x\in X\setminus\Sigma$. 

\smallskip

\par (D) $H^0_{(2)}(X\setminus\Sigma,L^p)\subset H^0(X,L^p)$ for every $p\geq1$.

\medskip

\par This variant of Theorem \ref{T:mt} will be useful to us for some applications in Section \ref{S:ex}, where the fact that the sections in $H^0_{(2)}(X\setminus\Sigma,L^p)$ extend holomorphically to sections of $L^p$ over $X$ is known to hold by other considerations (see Sections \ref{SS:qp} and \ref{SS:aq}).
\end{Remark}

\section{Distribution of zeros of random sections}\label{S:SZ}

\par The purpose of this section is to give the proof of Theorem \ref{T:SZ}. As a consequence we show in Theorem \ref{T:rs} that zeros of random holomorphic sections are equidistributed with respect to the curvature current.

\par Let $X,\Sigma,(L,h),f,\Omega$ verify assumptions (A)-(C) stated in the introduction. In addition, we assume in this section that $X$ is {\em compact}.  By Lemma \ref{L:gammap}, $H^0_{(2)}(X\setminus\Sigma,L^p)\subset H^0(X,L^p)$. Let 
$$d_p=\dim H^0_{(2)}(X\setminus\Sigma,L^p),\;\{S^p_j\}_{1\leq j\leq d_p} \text{ a fixed orthonormal basis of }H^0_{(2)}(X\setminus\Sigma,L^p).$$ 
The currents $\gamma_p$ can now be described as pullbacks $\gamma_p=\Phi_p^\star(\omega_{FS})$, where $\Phi_p:X\dashrightarrow{\mathbb P}^{d_p-1}$ is the Kodaira map defined by $\{S^p_j\}$ and $\omega_{FS}$ is  the Fubini-Study form  on ${\mathbb P}^{d_p-1}$. Recall that if $S_j^p=s_j^pe_\alpha^{\otimes p}$ where $e_\alpha$ is a holomorphic frame for $L$ on an open set $U_\alpha$ then 
$$\Phi_p(x)=[s_1^p(x):\ldots:s_{d_p}^p(x)],\;x\in U_\alpha.$$

\par Following the framework in \cite{ShZ99}, we identify the unit sphere ${\mathcal S}^p$ of $H^0_{(2)}(X\setminus\Sigma,L^p)$ to the unit sphere ${\mathbf S}^{2d_p-1}$ in ${\mathbb C}^{d_p}$ by 
$$a=(a_1,\dots,a_{d_p})\in{\mathbf S}^{2d_p-1}\longrightarrow S_a=\sum_{j=1}^{d_p}a_jS^p_j\in{\mathcal S}^p,$$ 
and we let $\lambda_p$ be the probabilty measure on ${\mathcal S}^p$ induced by the normalized surface measure on ${\mathbf S}^{2d_p-1}$, denoted also by $\lambda_p$ (i.e. $\lambda_p({\mathbf S}^{2d_p-1})=1$). We let $\lambda_p^k$ denote the product measure on $({\mathcal S}^p)^k$ determined by $\lambda_p$. Given a nontrivial section $S\in H^0(X,L^p)$ we denote by $[S=0]$ the current of integration (with multiplicities) over the analytic hypersurface $\{S=0\}$ of $X$. 

\medskip 

\par We give now the proof of Theorem \ref{T:SZ}. Let us note that some of the main ideas involved in proving this theorem are similar to those in \cite{ShZ99,ShZ08}, however special attention has to be given as we have to work with currents rather than smooth forms and the subspaces of sections we consider have nonempty base locus. To prove assertion $(i)$ we will need the following version of Bertini's theorem:

\begin{Proposition}\label{P:Bertini} Let $L\longrightarrow X$ be a holomorphic line bundle over a compact complex manifold $X$ of dimension $n$. Assume that: 

\par (i) $V$ is a vector subspace of $H^0(X,L)$ with basis $S_1,\dots, S_d$, and with base locus $Bs(V)=\{S_1=\ldots=S_d=0\}\subset X$ so that $\dim Bs(V)\leq n-k$.

\par (ii) $Z(h):=\{x\in X:\,\sum_{j=1}^dh_jS_j(x)=0\}$, where $h=[h_1:\ldots:h_d]\in{\mathbb P}^{d-1}$.

\par (iii) $\nu_l$ is the product measure on $({\mathbb P}^{d-1})^l$ induced by the Fubini-Study volume $\mu_{d-1}$ on ${\mathbb P}^{d-1}$. 

\par Then $\dim Z(h^1)\cap\ldots\cap Z(h^k)=n-k$, for $\nu_k$-a.e. $(h^1,\dots,h^k)\in({\mathbb P}^{d-1})^k$. 
\end{Proposition}

\par The proof is included at the end of this section for the convenience of the reader, since we could not find it in the literature. Assertion $(ii)$ of Theorem \ref{T:SZ} is proved by repeated application of the following proposition:

\begin{Proposition}\label{P:SZ} Let $L\longrightarrow X,\,V,\,S_1,\dots,S_d$, be as in Proposition \ref {P:Bertini}. Assume that:

\par (i) ${\mathcal S}:=\left\{\sum_{j=1}^da_jS_j:\,\sum_{j=1}^d|a_j|^2=1\right\}$ is endowed with the probability measure $\lambda$ induced via  the natural identification by the normalized surface measure on ${\mathbf S}^{2d-1}$. 

\par (ii) $\beta:=\Phi^\star(\omega_{FS})$, where $\Phi:X\dashrightarrow{\mathbb P}^{d-1}$ is the Kodaira map defined by $\{S_j\}$. 

\par (iii) $T$ is a positive closed current on $X$ of bidimension $(l,l)$, $l>0$, such that the current $[S=0]\wedge T$ is well defined for $\lambda$-a.e. $S\in{\mathcal S}$. 

\par Then the current $\beta\wedge T$ is well defined on $X$. Moreover, if $\varphi$ is a smooth $(l-1,l-1)$ form on $X$ the function $S\to\langle[S=0]\wedge T,\varphi\rangle$ is in $L^1({\mathcal S},\lambda)$ and 
$$\int_{\mathcal S}\langle[S=0]\wedge T,\varphi\rangle\,d\lambda(S)=\langle\beta\wedge T,\varphi\rangle.$$
\end{Proposition}

\par We postpone for the time being the proof of Proposition \ref{P:SZ}, and we continue with the proof of Theorem \ref{T:SZ}. 

\medskip

\par\noindent {\em Proof of Theorem \ref{T:SZ}.} $(i)$ Lemma \ref{L:gammapk} (and its proof) show that 
$$\dim Bs(H^0_{(2)}(X\setminus\Sigma,L^p))\leq n-k,$$ 
for all $p$ sufficiently large. It follows from Proposition \ref{P:Bertini} that the analytic subset $\{\sigma_1=0\}\cap\ldots\cap\{\sigma_k=0\}$ has pure dimension $n-k$ for $\lambda_p^k$-a.e. $\sigma=(\sigma_1,\ldots,\sigma_k)\in({\mathcal S}^p)^k$. Hence for such $\sigma$, the set $\{\sigma_{i_1}=0\}\cap\ldots\cap\{\sigma_{i_l}=0\}$ has pure dimension $n-l$ for every $i_1<\ldots<i_l$ in $\{1,\dots,k\}$. Therefore the current $[\sigma=0]$ is well defined \cite[Corollary 2.11]{D93}, and it equals the current of intersection with multiplicities along $\{\sigma_1=0\}\cap\ldots\cap\{\sigma_k=0\}$ \cite[Proposition 2.12]{D93}.

\smallskip

\par $(ii)$ If $\sigma=(\sigma_1,\ldots,\sigma_k)\in({\mathcal S}^p)^k$ is so that the set $\{\sigma_1=0\}\cap\ldots\cap\{\sigma_k=0\}$ has pure dimension $n-k$, Corollary 2.11 in \cite{D93} and the considerations from $(i)$ show that $[\sigma_{i_1}=0]\wedge\ldots\wedge[\sigma_{i_l}=0]\wedge\gamma_p$ is a well defined positive closed current of bidegree $(l+1,l+1)$ on $X$, for every $i_1<\ldots<i_l$ in $\{1,\dots,k\}$, $l<k$. 

\par By adding to $\varphi$ a large multiple of $\Omega^{n-k}$ we may assume that $\varphi$ is a strongly positive $(n-k,n-k)$ test form on $X$. Hence the integral in $(ii)$ can be evaluated as an iterated integral by Tonelli's theorem. We apply Proposition \ref{P:SZ} with 
$$V=H^0_{(2)}(X\setminus\Sigma,L^p),\;T=[\sigma_2=0]\wedge\ldots\wedge[\sigma_k=0].$$ 
Then for $\lambda_p^{k-1}$-a.e. $(\sigma_2,\ldots,\sigma_k)\in({\mathcal S}^p)^{k-1}$,
$$\int_{{\mathcal S}^p}\langle[\sigma=0],\varphi\rangle\,d\lambda_p(\sigma_1)=\langle T\wedge\gamma_p,\varphi\rangle=\langle[\sigma_2=0]\wedge\ldots\wedge[\sigma_k=0]\wedge\gamma_p,\varphi\rangle,$$
since $[\sigma=0]=[\sigma_1=0]\wedge T$. Proposition \ref{P:Bertini} shows that Proposition \ref{P:SZ} can be applied again for $\lambda_p^{k-2}$-a.e. $(\sigma_3,\ldots,\sigma_k)\in({\mathcal S}^p)^{k-2}$ with $T=[\sigma_3=0]\wedge\ldots\wedge[\sigma_k=0]\wedge\gamma_p$, so  
\begin{eqnarray*}
\int_{{\mathcal S}^p}\int_{{\mathcal S}^p}\langle[\sigma=0],\varphi\rangle\,d\lambda_p(\sigma_1)\,d\lambda_p(\sigma_2)&=&\int_{{\mathcal S}^p}\langle[\sigma_2=0]\wedge\ldots\wedge[\sigma_k=0]\wedge\gamma_p,\varphi\rangle\,d\lambda_p(\sigma_2)\\&=&\langle[\sigma_3=0]\wedge\ldots\wedge[\sigma_k=0]\wedge\gamma_p^2,\varphi\rangle.
\end{eqnarray*}
Continuing like this we obtain that the $k$-th iterated integral in $(ii)$ equals $\langle\gamma_p^k,\varphi\rangle$. This proves $(ii)$, and then $(iii)$ follows at once from Theorem \ref{T:mt}. $\qed$

\medskip

\par Let us now consider the probability space ${\mathcal S}_\infty=\prod_{p=1}^\infty{\mathcal S}^p$ endowed with the probability measure $\lambda_\infty=\prod_{p=1}^\infty\lambda_p$. The proof of the variance estimate from Lemma 3.3 in \cite{ShZ99} goes through with no change. Combined with Theorem \ref{T:SZ}, it implies that Theorem 1.1 of \cite{ShZ99} holds in our setting. Namely, we have the following: 

\begin{Theorem}\label{T:rs} In the setting of Theorem \ref {T:mt}, assume that $X$ is compact and that \eqref{e:mainhyp} holds. Then, in the weak sense of currents on $X$,
$$\lim_{p\to\infty}\,\frac{1}{p}\,[\sigma_p=0]=\gamma,\;for\;\lambda_\infty\text{-a.e. sequence }\{\sigma_p\}_{p\geq1}\in{\mathcal S}_\infty.$$ 
\end{Theorem}

\medskip

\par\noindent{\em Proof of Proposition \ref{P:SZ}.} We fix a holomorphic frame $e_\alpha$ of $L$ over an open set $U_\alpha$, and write $S=se_\alpha$, $S_j=s_je_\alpha$, where $S=\sum_{j=1}^da_jS_j\in{\mathcal S}$ is chosen so that the current $[S=0]\wedge T$ is well defined. Then 
$$\log|s|=\log\left|\sum_{j=1}^da_js_j\right|\leq\frac{1}{2}\,\log\left(\sum_{j=1}^d|s_j|^2\right).$$ 
Since the latter function is locally bounded above on $U_\alpha$ and $\log|s|\in L^1_{loc}(U_\alpha,|T|)$, we conclude that $\log(\sum_{j=1}^d|s_j|^2)\in L^1_{loc}(U_\alpha,|T|)$, so $\beta\wedge T$ is well defined. 

\par For $S\in{\mathcal S}$ we define the function $N(S)$ on $X$ by 
$$N(S)\mid_{_{U_\alpha}}=\log\frac{|s|}{\sqrt{|s_1|^2+\ldots+|s_d|^2}}\;.$$
Note that $N(S)\in L^1(X,\Omega^n)$, where $\Omega$ is a smooth positive $(1,1)$ form on $X$, since it is locally the difference of psh functions. Moreover, for $\lambda$-a.e. $S\in{\mathcal S}$, $N(S)\in L^1(X,|T|)$. Therefore we have 
$$[S=0]=\beta+dd^cN(S),\;[S=0]\wedge T=\beta\wedge T+dd^c(N(S)T).$$
Indeed, the first formula follows from the definition of the function $N(S)$, while for the second, working locally on $U_\alpha$, we have 
\begin{eqnarray*}
[S=0]\wedge T&=&dd^c(\log|s|\,T)=dd^c\left(\log\sqrt{|s_1|^2+\ldots+|s_d|^2}\;T\right)+dd^c(N(S)T)\\
&=&\beta\wedge T+dd^c(N(S)T).
\end{eqnarray*}
Thus, for $\lambda$-a.e. $S\in{\mathcal S}$, 
$$\langle[S=0]\wedge T,\varphi\rangle=\langle\beta\wedge T,\varphi\rangle+\int_XN(S)\,T\wedge dd^c\varphi,$$
and the proof is finished if we show that the function $S\to\int_XN(S)\,T\wedge dd^c\varphi$ belongs to $L^1({\mathcal S},\lambda)$ and 
$$\int_{{\mathcal S}}\left(\int_XN(S)\,T\wedge dd^c\varphi\right)d\lambda(S)=0\,.$$

\par We may assume that $\varphi$ is real, so $dd^c\varphi$ is a real $(l,l)$ from on $X$. There exists a constant $M$ so that $dd^c\varphi+M\Omega^l$ is a strongly positive $(l,l)$ form, so $T\wedge(dd^c\varphi+M\Omega^l)$ is a positive measure. It follows that we can write $T\wedge dd^c\varphi=\mu_1-\mu_2$, where $\mu_j$ are positive measures and $N(S)\in L^1(X,\mu_j)$ for $\lambda$-a.e. $S\in{\mathcal S}$. Note also that $N(S)\leq0$ on $X$. Hence by Tonelli's theorem,
$$\int_{{\mathcal S}}\left(\int_XN(S)\,d\mu_j\right)d\lambda(S)=\int_X\left(\int_{{\mathcal S}}N(S)\,d\lambda(S)\right)d\mu_j\,.$$

\par Since on $U_\alpha$ the function $\log(|s_1|^2+\ldots+|s_d|^2)$ is locally integrable with respect to $|T|$, hence with respect to $\mu_j$, we have $|s_1|^2+\ldots+|s_d|^2>0$ $\mu_j$-a.e. on $U_\alpha$. So  
$$u_\alpha:=\left(\frac{s_1}{\sqrt{|s_1|^2+\ldots+|s_d|^2}}\,,\ldots,\frac{s_d}{\sqrt{|s_1|^2+\ldots+|s_d|^2}}\right),$$
is a well defined function $\mu_j$-a.e. on $U_\alpha$ with values in the unit sphere ${\mathbf S}^{2d-1}$ in ${\mathbb C}^d$. We have 
$$N(S)=N(a_1S_1+\ldots+a_dS_d)=\log|a\cdot u_\alpha| \text{ on } U_\alpha,$$ 
where $a=(a_1,\dots,a_d)$ and $a\cdot u=a_1u_1+\ldots+a_du_d$. Therefore 
$$\int_{{\mathcal S}}N(S)(x)\,d\lambda(S)=\int_{{\mathbf S}^{2d-1}}\log|a\cdot u_\alpha(x)|\,d\lambda(a)=C_d,$$
for $\mu_j$-a.e. $x\in U_\alpha$, where $C_d<0$ is a dimensional constant. It follows that 
$$\int_{{\mathcal S}}\left(\int_XN(S)\,d\mu_j\right)d\lambda(S)=C_d\mu_j(X)>-\infty,$$
so the function $S\to\int_XN(S)\,d\mu_j$ is in $L^1({\mathcal S},\lambda)$, hence so is the function 
$$S\to\int_XN(S)\,T\wedge dd^c\varphi=\int_XN(S)\,d\mu_1-\int_XN(S)\,d\mu_2.$$
Finally, since $T$ is closed we have 
$$\int_{{\mathcal S}}\left(\int_XN(S)\,T\wedge dd^c\varphi\right)d\lambda(S)=C_d(\mu_1(X)-\mu_2(X))=C_d\int_X T\wedge dd^c\varphi=0.$$
This concludes the proof. $\qed$

\medskip

\par\noindent{\em Proof of Proposition \ref{P:Bertini}.} We divide the proof in three steps.

\smallskip

\par {\em Step 1.} We show that for $\nu_{k-1}$-a.e. $(h^1,\dots,h^{k-1})\in({\mathbb P}^{d-1})^{k-1}$ the set 
$$Z(h^1)\cap\ldots\cap Z(h^{k-1})\cap(X\setminus Bs(V))$$ 
is a complex submanifold of $X\setminus Bs(V)$ of dimension $n-k+1$. 

\par Consider the set ${\mathcal I}\subset(X\setminus Bs(V))\times({\mathbb P}^{d-1})^{k-1}$ defined by 
$$(x,h^1,\dots,h^{k-1})\in{\mathcal I}\Longleftrightarrow\sum_{j=1}^dh_j^iS_j(x)=0,\;1\leq i\leq k-1,$$
where $h^i=[h_1^i:\dots:h_d^i]$. If $z=(x,h^1,\dots,h^{k-1})\in{\mathcal I}$ then $x\not\in Bs(V)$, and we may assume that for each $i$, $h^i_{j_i}=1$ for some $1\leq j_i\leq d$. For each $i$ there exists $l_i\neq j_i$ so that $S_{l_i}(x)\neq0$. Indeed, otherwise $S_l(x)=0$ for all $l\neq j_i$, so 
$$S_{j_i}(x)=\sum_{l\neq j_i}h^i_lS_l(x)+S_{j_i}(x)=0,$$
hence $x\in Bs(V)$, a contradiction. We obtain that for $z'=(x',\zeta^1,\dots,\zeta^{k-1})$ near $z$, ${\mathcal I}$ is the graph
$$\zeta^i_{l_i}=-\frac{S_{j_i}(x')}{S_{l_i}(x')}-\sum_{l\neq l_i,j_i}\frac{S_l(x')}{S_{l_i}(x')}\,\zeta^i_l\,,\;1\leq i\leq k-1.$$
Thus ${\mathcal I}$ is a submanifold of $(X\setminus Bs(V))\times({\mathbb P}^{d-1})^{k-1}$ of dimension $n+(k-1)(d-2)$. 

\par We claim that the projection 
$$\pi_2:{\mathcal I}\longrightarrow({\mathbb P}^{d-1})^{k-1},\;\pi_2(x,h^1,\dots,h^{k-1})=(h^1,\dots,h^{k-1}),$$ 
is surjective. Indeed, $Z(h^i)\neq X$ since $(S_1,\dots,S_d)$ is a basis of $V$, so $\dim Z(h^1)\cap\ldots\cap Z(h^{k-1})\geq n-k+1$. As $\dim Bs(V)\leq n-k$, we can find $x\in Z(h^1)\cap\ldots\cap Z(h^{k-1})\cap(X\setminus Bs(V))$, so  $(x,h^1,\dots,h^{k-1})\in{\mathcal I}$. 

\par By Sard's theorem, for $\nu_{k-1}$-a.e. $(h^1,\dots,h^{k-1})\in({\mathbb P}^{d-1})^{k-1}$ the set
$$\pi_2^{-1}(h^1,\dots,h^{k-1})=\{(x,h^1,\dots,h^{k-1}):\,x\in Z(h^1)\cap\ldots\cap Z(h^{k-1})\cap(X\setminus Bs(V))\}$$
is a submanifold of ${\mathcal I}$ of dimension $\dim{\mathcal I}-(k-1)(d-1)=n-k+1$. Since $\pi_1:(X\setminus Bs(V))\times\{(h^1,\dots,h^{k-1})\}\longrightarrow X\setminus Bs(V)$ is a biholomorphism, we conclude that $Z(h^1)\cap\ldots\cap Z(h^{k-1})\cap(X\setminus Bs(V))$ is a submanifold of $X\setminus Bs(V)$ of dimension $n-k+1$. In particular, $Z(h^1)\cap\ldots\cap Z(h^{k-1})$ is analytic subset of $X$ of pure dimension $n-k+1$, smooth away from $Bs(V)$. 

\smallskip

\par {\em Step 2.} We show that the set $G_k$ is open, where 
$$G_k=\{(h^1,\dots,h^k)\in({\mathbb P}^{d-1})^k:\,\dim Z(h^1)\cap\ldots\cap Z(h^k)=n-k\}.$$
Indeed, assume for a contradiction that $(h^1,\dots,h^k)\in G_k$ but there exist sequences $h^i_N\to h^i$ in ${\mathbb P}^{d-1}$, as $N\to\infty$, so that the set $Z(h_N^1)\cap\ldots\cap Z(h_N^k)$ has an irreducible component $A_N$ of dimension $m$, for some $m>n-k$. Consider the currents $T_N=({\rm vol}\,A_N)^{-1}[A_N]$, where $[A_N]$ is the current of integration on $A_N$. Since $T_N$ have unit mass, we may assume by passing to a subsequence that $T_N$ converge weakly to a positive closed current $T$ of unit mass and bidimension $(m,m)$. Note that the sets $A_N$ cluster to the analytic set $A=Z(h^1)\cap\ldots\cap Z(h^k)$, so $T$ is supported on $A$. Since $\dim A=n-k<m$, $T=0$ by the support theorem, a contradiction. 

\smallskip

\par {\em Step 3.} We show that the complement $G_k^c=({\mathbb P}^{d-1})^k\setminus G_k$ has $\nu_k$ measure 0. Let 
$$G_{k-1}=\{(h^1,\dots,h^{k-1})\in({\mathbb P}^{d-1})^{k-1}:\,\dim Z(h^1)\cap\ldots\cap Z(h^{k-1})=n-k+1\}.$$
By steps 1 and 2, the set $G_{k-1}$ is open and $\nu_{k-1}(G_{k-1}^c)=0$. We have 
$$G_k^c\subset\left(G_{k-1}^c\times{\mathbb P}^{d-1}\right)\cup\left(G_k^c\cap(G_{k-1}\times{\mathbb P}^{d-1})\right).$$
Note that $\nu_k(G_{k-1}^c\times{\mathbb P}^{d-1})=0$ and the set $F_k=G_k^c\cap(G_{k-1}\times{\mathbb P}^{d-1})$ is $\nu_k$ measurable. 

\par For $(h^1,\dots,h^{k-1})\in G_{k-1}$ consider the slice 
\begin{eqnarray*}
F_k(h^1,\dots,h^{k-1})&=&\{h\in{\mathbb P}^{d-1}:\,(h^1,\dots,h^{k-1},h)\in F_k\}\\
&=&\{h\in{\mathbb P}^{d-1}:\,(h^1,\dots,h^{k-1},h)\in G_k^c\}.
\end{eqnarray*}
Since $G_k^c$ is closed, the above slices are closed. We are done if we show that they have $\mu_{d-1}$ measure 0. Indeed, since $F_k$ is measurable this will imply that $\nu_k(F_k)=0$. 

\par To this end we let $Y:=Z(h^1)\cap\ldots\cap Z(h^{k-1})=Y_1\cup\ldots\cup Y_N$, where $Y_l$ are the irreducible components of $Y$. Since all of them have dimension $n-k+1$ it follows that 
$$F_k(h^1,\dots,h^{k-1})=\bigcup_{j=1}^NE_j\,,\;E_j:=\{h\in{\mathbb P}^{d-1}:\,Y_j\subset Z(h)\}.$$
Note that the sets $E_j$ are closed. We will be done if we show that $\mu_{d-1}(E_j)=0$. 

\par Let us fix $j$. The basis sections $S_i$ cannot all vanish identically on $Y_j$, since $\dim Y_j=n-k+1$ and $\dim Bs(V)\leq n-k$. We may assume that $S_d\not\equiv0$ on $Y_j$. So 
$$E_j\subset\{\zeta_1=0\}\cup H_j\,,\;H_j:=\{(\zeta_2,\dots,\zeta_d)\in{\mathbb C}^{d-1}:\,[1:\zeta_2:\ldots:\zeta_d]\in E_j\}.$$
Note that $H_j$ is closed in ${\mathbb C}^{d-1}$, and we are done if we show that it has Lebesgue measure 0. This follows since for each $(\zeta_2,\dots,\zeta_{d-1})\in{\mathbb C}^{d-2}$ the slice 
$$H_j(\zeta_2,\dots,\zeta_{d-1})=\{\zeta\in{\mathbb C}:\,(\zeta_2,\dots,\zeta_{d-1},\zeta)\in H_j\}$$ 
contains at most one element. Indeed, if $\zeta\neq\zeta'\in H_j(\zeta_2,\dots,\zeta_{d-1})$ then 
$$S_1+\zeta_2S_2+\ldots+\zeta_{d-1}S_{d-1}+\zeta S_d\equiv0\,,\;S_1+\zeta_2S_2+\dots+\zeta_{d-1}S_{d-1}+\zeta'S_d\equiv0$$
on $Y_j$, hence $S_d\equiv0$ on $Y_j$, a contradiction. $\qed$

\section{Asymptotic behavior of the Bergman kernel function}\label{S:Bkf}

\par Using techniques of Demailly from \cite[Proposition 3.1]{D92}, \cite[Section 9]{D93b} we prove here two theorems about the asymptotic behavior of the Bergman kernel function. The first one hereafter holds for arbitrary singular metrics with strictly positive curvature, while the second one, Theorem \ref{T:Bkf},  shows that our hypothesis \eqref{e:mainhyp} is satisfied in a quite general setting.

\begin{Theorem}\label{T:BkfK} Let $(X,\Omega)$ be a compact $n$-dimensional K\"ahler manifold and $(L,h)$ be a holomorphic line bundle on $X$ with a singular Hermitian metric $h$ so that $c_1(L,h)$ is a strictly positive current. If $P_p,\gamma_p$ are the Bergman kernel function, resp. the Fubini-Study currents, defined by \eqref{e:Bergfcn}-\eqref{e:gammap} for the spaces $H^0_{(2)}(X,L^p)$ of $L^2$-holomorphic sections of $L^p$ relative to the metric induced by $h$ and the volume form  $\Omega^n$, then as $p\to\infty$, 
$$\frac{1}{p}\,\log P_p\to0\;\,in\;\,L^1(X,\Omega^n)\,,\;\,\frac{1}{p}\,\gamma_p\to c_1(L,h)\,,\;\,\frac{1}{p}\,[\sigma_p=0]\to c_1(L,h)\,$$
for $\lambda_\infty$-a.e. sequence $\{\sigma_p\}_{p\geq1}\in{\mathcal S}_\infty$, in the weak sense of currents on $X$, where $\mathcal{S}_\infty,\lambda_\infty$ are as in Theorem \ref{T:rs}.
\end{Theorem}

\par We will need the following existence theorem for $\overline\partial$ in the case of singular Hermitian metrics. The smooth case goes back to Andreotti-Vesentini and H\"ormander, while the singular case was first observed by Bombieri and Skoda and proved in generality by Demailly \cite[Theorem 5.1]{D82}.

\begin{Theorem}[$L^2$-estimates for $\overline\partial$]\label{T:l2}
Let $(M,\Theta)$ be a complete K\"ahler manifold, $(L,h)$ be a singular Hermitian line bundle and $\varphi$ a quasi-psh function on $M$. Assume that there exist constants $a>0$, $C>0$ such that
$$c_1(L,h)>2a\Theta,\quad dd^{c}\varphi>-C\Theta,\quad c_1(K_M,h^{K_M})<C\Theta\,,$$
where $h^{K_M}$ is the metric induced on $K_M$ by $\Theta$. Then there exists $p_0=p_0(a,C)$ such that for any $p\geq p_0$ and for any form $g\in L_{0,1}^2(M,L^p)$ satisfying  
${\overline\partial}g=0$ there exists $u\in L_{0,0}^2(M,L^p)$ with $\overline\partial u=g$ and  
$$\int_M|u|^2_{h_p}e^{-\varphi}\,dv_M\leq\frac{1}{ap}\int_M|g|^2_{h_p}e^{-\varphi}\,dv_M\,.$$
\end{Theorem}

\medskip

\par\noindent{\em Proof of Theorem \ref{T:BkfK}.} Let $x\in X$ and $U_\alpha\subset X$ be a coordinate neighborhood of $x$ on which there exists a holomorphic frame $e_\alpha$ of $L$. Let $\psi_\alpha$ be a psh weight of $h$ on $U_\alpha$. Fix $r_0>0$ so that the ball $V:=B(x,2r_0)\subset\subset U_\alpha$ and let $U:=B(x,r_0)$.

\par We show that there exist constants $C_1>0$, $p_0\in\mathbb{N}$ so that 
\begin{equation}\label{e:Bke}
-\frac{\log C_1}{p}\leq\frac{1}{p}\,\log P_p(z)\leq\frac{\log(C_1r^{-2n})}{p}+2\left(\max_{B(z,r)}\psi_\alpha-\psi_\alpha(z)\right)
\end{equation}
holds for all $p>p_0$, $0<r<r_0$ and $z\in U$ with $\psi_\alpha(z)>-\infty$.  

\smallskip

\par For the upper estimate, fix $z\in U$ with $\psi_\alpha(z)>-\infty$ and $r<r_0$. Let $S\in H^0_{(2)}(X,L^p)$ with $\|S\|_p=1$ and write $S=se_\alpha^{\otimes p}$. Repeating an argument of Demailly we obtain 
\begin{eqnarray*}
|S(z)|^2_{h_p}&=&|s(z)|^2e^{-2p\psi_\alpha(z)}\leq e^{-2p\psi_\alpha(z)}\frac{C_1}{r^{2n}}\,\int_{B(z,r)}|s|^2\,\Omega^n\\
&\leq&\frac{C_1}{r^{2n}}\,\exp\left(2p\left(\max_{B(z,r)}\psi_\alpha-\psi_\alpha(z)\right)\right)\int_{B(z,r)}|s|^2e^{-2p\psi_\alpha}\,\Omega^n\\
&\leq&\frac{C_1}{r^{2n}}\,\exp\left(2p\left(\max_{B(z,r)}\psi_\alpha-\psi_\alpha(z)\right)\right),
\end{eqnarray*}
where $C_1$ is a constant that depends only on $x$. Hence  
$$\frac{1}{p}\,\log P_p(z)=\frac{1}{p}\,\max_{\|S\|_p=1}\log |S(z)|^2_{h_p}\leq\frac{\log(C_1r^{-2n})}{p}+2\left(\max_{B(z,r)}\psi_\alpha-\psi_\alpha(z)\right).$$
Note that this estimate holds for all $p$ and it does not require the strict positivity of the current $c_1(L,h)$, nor the hypotheses that $X$ is compact or $\Omega$ is a K\"ahler form. 

\smallskip

\par For the lower estimate in \eqref{e:Bke}, we proceed as in \cite[Section 9]{D93b} to show that there exist a constant $C_1>0$ and $p_0\in\mathbb{N}$ such that for all $p>p_0$ and all $z\in U$ with $\psi_\alpha(z)>-\infty$ there is a section $S_{z,p}\in H^0_{(2)}(X,L^p)$ with $S_{z,p}(z)\neq0$ and 
$$\|S_{z,p}\|^2_p\leq C_1|S_{z,p}(z)|^2_{h_p}\,.$$
Observe that this implies that 
$$\frac{1}{p}\,\log P_p(z)=\frac{1}{p}\,\max_{\|S\|_p=1}\log|S(z)|^2_{h_p}\geq-\frac{\log C_1}{p}\,.$$

\par Let us prove the existence of $S_{z,p}$ as above. By the Ohsawa-Takegoshi extension theorem \cite{OT87} there exists a constant $C'>0$ (depending only on $x$) such that for any $z\in U$ and any $p$ there exists a function $v_{z,p}\in{\mathcal O}(V)$ with $v_{z,p}(z)\neq0$ and 
$$\int_V|v_{z,p}|^2e^{-2p\psi_\alpha}\Omega^n\leq C'|v_{z,p}(z)|^2e^{-2p\psi_\alpha(z)}\,.$$

\par We shall now solve the $\overline\partial$-equation with $L^2$-estimates in order to extend $v_{z,p}$ to a section of $L^p$ over $X$. We apply Theorem \ref{T:l2} for $(X,\Omega)$ and $(L,h)$.
Let $\theta\in\mathcal C^\infty(\mathbb R)$ be a cut-off function such that $0\leq\theta\leq1$, $\theta(t)=1$ for $|t|\leq\frac12$, $\theta(t)=0$ for $|t|\geq1$.
Define the quasi-psh function $\varphi_z$ on $X$ by
\[
\varphi_z(y)=\begin{cases}2n\theta\big(\tfrac{|y-z|}{r_0}\big)\log\frac{|y-z|}{r_0}\,,\quad\text{for $y\in U_\alpha$}\,,\\
0,\quad\text{for $y\in X\setminus B(z,r_0)$}\,.
\end{cases}
\]
Note that there exist $a>0$, $C>0$ such that the hypotheses of Theorem \ref{T:l2} are satisfied for $(X,\Omega)$, $(L,h)$ and all $\varphi_z$, $z\in U$. Let $p_0$ be as in Theorem \ref{T:l2}. Consider the form $$g\in L^2_{0,1}(X,L^p),\;g=\overline\partial\big(v_{z,p}\theta\big(\tfrac{|y-z|}{r_0}\big)e_\alpha^{\otimes p}\big).$$
By Theorem \ref{T:l2}, for each $p>p_0$ there exists $u\in L^2_{0,0}(X,L^p)$ such that $\overline\partial u =g$ and
$$\int_X|u|^2_{h_p}e^{-\varphi_z}\,\Omega^n\leq\frac{1}{ap}\int_X|g|^2_{h_p}e^{-\varphi_z}\Omega^n<\infty\,.$$
Here the second integral is finite since $\psi_\alpha(z)>-\infty$ and 
$$\int_X|g|^2_{h_p}e^{-\varphi_z}\Omega^n=\int_V|v_{z,p}|^2|\overline\partial\theta(\tfrac{|y-z|}{r_0})|^{2}e^{-2p\psi_\alpha}e^{-\varphi_z}\Omega^n\leq C''\int_V|v_{z,p}|^2e^{-2p\psi_\alpha}\Omega^n,$$
where $C''>0$ is a constant that depends only on $x$. Near $z$, $e^{-\varphi_z(y)}=r_0^{2n}|y-z|^{-2n}$ is not integrable, thus $u(z)=0$. Define 
$$S_{z,p}:=v_{z,p}\theta\big(\tfrac{|y-z|}{r_0}\big)e_\alpha^{\otimes p}-u.$$
Then $\overline\partial S_{z,p}=0$, $S_{z,p}(z)=v_{z,p}(z)e_\alpha^{\otimes p}(z)\neq0$, $S_{z,p}\in H^0_{(2)}(X,L^p)$. Since $\varphi_z\leq0$ on $X$,
\begin{eqnarray*}
\|S_{z,p}\|^2_p&\leq&2\left(\int_V|v_{z,p}|^2e^{-2p\psi_\alpha}\Omega^n+\int_X|u|^2_{h_p}e^{-\varphi_z}\,\Omega^n\right)\\
&\leq&2C'\left(1+\frac{C''}{ap}\right)|v_{z,p}(z)|^2e^{-2p\psi_\alpha(z)}=C_1|S_{z,p}(z)|^2_{h_p},
\end{eqnarray*}
with a constant $C_1>0$ that depends only on $x$. This concludes the proof of \eqref{e:Bke}. 

\smallskip

\par Recall that $\log P_p\in L^1(X,\Omega^n)$, as it is locally the difference of psh functions. Observe that, by the upper semicontinuity of $\psi_\alpha$, \eqref{e:Bke} implies that $\frac{1}{p}\,\log P_p\to0$ as $p\to\infty$, $\Omega^n$-a.e. on $X$. Since $\psi_\alpha$ is psh on $U_\alpha$, it is integrable on $U$. By dominated convergence, \eqref{e:Bke} implies that $\frac{1}{p}\,\log P_p\to0$ in $L^1(U,\Omega^n)$, hence in $L^1(X,\Omega^n)$, so  
$$\gamma_p-c_1(L,h)=\frac{1}{2p}\,dd^c\log P_p\to0 \,\text{ weakly on } X.$$

\par The conclusion about the equidistribution of zeros of random sequences of sections now follows as in \cite[Theorem 1.1]{ShZ99} (see Section \ref{S:SZ} and Theorem \ref{T:rs}). $\qed$

\medskip

\par We return to the main setting of the paper, given by assumptions (A)-(C) stated in the introduction, and we take here $f\equiv1$.

\begin{Theorem}\label{T:Bkf} Let $X,\Sigma,(L,h),\Omega$ verify (A)-(B) and assume that $X$ is compact, $\Omega$ is a K\"ahler form, and $c_1(L,h)$ is a strictly positive current on $X$. Then \eqref{e:mainhyp} holds for the Bergman kernel function $P_p$ defined in \eqref{e:Bergfcn} for the space $H^0_{(2)}(X\setminus\Sigma,L^p)$.
\end{Theorem}

\begin{proof} Let $x\in X\setminus\Sigma$, $U_\alpha\subset X\setminus\Sigma$, $\psi_\alpha$, $V$, $U$, be as in the proof of Theorem \ref{T:BkfK}. Then \eqref{e:Bke} shows that $\frac{1}{p}\,\log P_p\to 0$ as $p\to\infty$ uniformly on $U$, thanks to the uniform continuity of $\psi_\alpha$ on $V$.
\end{proof}

\par Combining Theorems \ref{T:mt}, \ref{T:Bkf} and \ref{T:SZ} we obtain the following equidistribution theorem for big line bundles:

\begin{Theorem}\label{T:mt2} Let $(L,h)$ be a line bundle over the compact K\"ahler manifold $(X,\Omega)$ endowed with a singular Hermitian metric $h$ which is continuous outside a proper analytic subset $\Sigma$ and so that $\gamma:=c_1(L,h)$ is a strictly positive current. If $\gamma_p$ is the current defined by \eqref{e:gammap} for the space $H^0_{(2)}(X\setminus\Sigma,L^p)$ then $\frac{1}{p}\,\gamma_p\to\gamma$ weakly on $X$. If $\dim\Sigma\leq n-k$ for some $2\leq k\leq n$, then the currents $\gamma^k$ and $\gamma_p^k$, for all $p$ sufficiently large, are well defined and $\frac{1}{p^k}\,\gamma_p^k\to\gamma^k$ weakly on $X$. Moreover, the conclusions of Theorems \ref{T:SZ} and \ref{T:rs} hold in this setting.
\end{Theorem}

\par Note that in Theorems \ref{T:BkfK} and \ref{T:mt2} the bundle $L$ is a big line bundle and $X$ is Moishezon, by a theorem of Ji and Shiffman \cite{JS93} (cf.\ also \cite[Th.\,2.3.28,\,2.3.30]{MM07}). Hence $X$ is in fact a projective manifold, since it is assumed to be K\"ahler (see e.\,g.\ \cite[Th.\,2.2.26]{MM07}).

\section{Applications}\label{S:ex}

\par Let $X,\Sigma,(L,h),f,\Omega$ verify assumptions (A)-(C) stated in the introduction and assume in addition that $\gamma=c_1(L,h)$ is a strictly positive current. To emphasize the metrics that are used, we denote throughout this section the corresponding spaces of $L^2$-holomorphic sections by $H^0_{(2)}(X\setminus\Sigma,L^p,h,f\Omega^n)$. We discuss here several important situations in which the Bergman kernel function $P_p$ defined in \eqref{e:Bergfcn} satisfies our hypothesis \eqref{e:mainhyp}. In Sections \ref{SS:heps}, \ref{SS:sm} we consider singular Hermitian metrics on big line bundles, and we deduce equidistribution results for $L^2$ holomorphic sections with respect to the Poincar\'e metric and for sections of Nadel multiplier sheaves. In Section \ref{SS:qp} we turn to Zariski-open manifolds with bounded negative Ricci curvature, and we generalize a theorem of Tian \cite[Theorem C]{Ti90} in our framework. Natural examples of K\"ahler-Einstein manifolds of negative Ricci curvature are the arithmetic quotients. We show in Section \ref{SS:aq} how our results apply for toroidal compactifications of such manifolds. In Section \ref{SS:rmad} we point out what simplifications occur in the case of adjoint bundles.
Finally, in Sections \ref{SS:1conv}, \ref{SS:spscd} we exhibit some results for $1$-convex manifolds.

\subsection{Properties of $h_\varepsilon$}\label{SS:heps} 
For some of the applications, we will have to work with the Poincar\'e metric $\Theta$ on $X\setminus\Sigma$ and with a small perturbation $h_\varepsilon$ of the metric $h$ on $L$. Let us begin by listing certain properties of these special metrics. 

We refer to Section \ref{SS:metrics} for the construction of the metrics $\Theta,\,h_\varepsilon$, and we shall use the notations introduced there. In particular, $\Theta^n=f\Omega^n$ with a function $f$ as in (C) (see Section \ref{SSS:Theta}). Note that $h_\varepsilon$ is in fact a metric on $L\mid_{_{X\setminus Y}}$, where $Y\subset\Sigma$ is an analytic subset of dimension $\leq n-2$ (Section \ref{SSS:heps}). We recall the following fact:

\begin{Lemma}\label{L:hext} Let $L$ be a holomorphic line bundle over a complex manifold $X$ and $Y$ be an analytic subvariety of $X$ with $\codim Y\geq2$. Then any positively curved singular metric $h^L$ on $L\mid_{_{X\setminus Y}}$ extends to a positively curved singular metric on $L$. Moreover, if $c_1(L\mid_{_{X\setminus Y}},h^L)\geq\delta\Omega$ on $X\setminus Y$, for some $\delta>0$, then the same estimate holds for the curvature current of the extended metric on $X$. 
\end{Lemma}

\begin{proof} If $U_\alpha$ is a neighborhood of some point $y\in Y$ on which $L$ has a holomorphic frame $e_\alpha$, then $h^L(e_\alpha,e_\alpha)=e^{-2\varphi_\alpha}$ for some psh function $\varphi_\alpha$ on $U_\alpha\setminus Y$. Since $\codim Y\geq2$ the function $\varphi_\alpha$ is locally upper bounded near the points of $U_\alpha\cap Y$, hence it extends to a psh function on $U_\alpha$. The second conclusion follows  since the current $c_1(L,h^L)$ does not charge $Y$ by Federer's support theorem (\cite{Fed69}, see also \cite[Theorem 1.7]{Har77}). 
\end{proof} 

\par We denote the extended metric still by $h_\epsilon$ and we let $\omega=c_1(L,h_\varepsilon)$, so $\omega$ is a positive closed $(1,1)$ current on $X$. 

\begin{Proposition}\label{P:heps} (i) We have $ \omega=\gamma+\pi_\star(\theta'+\varepsilon dd^cF)$, 
where $F$ is defined in \eqref{e:F} and $\theta'$ is a smooth real closed $(1,1)$ form on $\widetilde X$.

\par (ii) Let $A$ be an irreducible component of $\Sigma$ of dimension $n-1$.  Then the generic Lelong numbers $\nu(\gamma,A)=\nu(\omega,A)$. Moreover, any section in $H^0_{(2)}(X\setminus\Sigma,L^p,h_\varepsilon,f\Omega^n)$ vanishes at least to order $p\nu(\omega,A)$ on $A$.
\end{Proposition}

\begin{proof} $(i)$ Recall from Section \ref{SSS:heps} that the metric $h_\varepsilon$ on $L\mid_{_{X\setminus Y}}$ was induced via the biholomorphism $\pi:\widetilde X\setminus E\longrightarrow X\setminus Y$ by a metric $h_\varepsilon^{L'}$ on $L'=\pi^\star\left(L\mid_{_{X\setminus Y}}\right)$ with curvature current $\gamma'_\varepsilon=\pi^\star\gamma+\theta'+\varepsilon dd^cF$. The map $\pi:\widetilde X\longrightarrow X$ is proper so $\pi_\star\gamma'_\varepsilon$ is a well defined positive closed (1,1) current on $X$ which satisfies $\pi_\star\gamma'_\varepsilon=\omega$ on $X\setminus Y$. As $\dim Y\leq n-2$, Federer's support theorem \cite{Fed69} implies that $\pi_\star\gamma'_\varepsilon=\omega$ on $X$. Similarly, we have that $\pi_\star\pi^\star\gamma=\gamma$ on $X\setminus Y$, and hence on $X$. The formula for $\omega$ in the statement now follows. 

\medskip

\par $(ii)$ Fix a point $x\in A\setminus Y$. Then $x\in\Sigma^{n-1}_{reg}$, so we can find a neighborhood $V_x\subset X$ of $x$ and local coordinates $z_1,\dots,z_n$ on $V_x$ so that $\pi:\pi^{-1}(V_x)\longrightarrow V_x$ is a biholomorphism, $x=0$, $\Sigma\cap V_x=A\cap V_x=\{z_1=0\}$, and $f\geq c>0$ on $V_x$. 

\par By $(i)$ we have 
$$\omega=\gamma+(\pi^{-1})^\star\theta'+\varepsilon dd^cF\circ\pi^{-1} \text{ on } V_x.$$ 
We can assume that there exist functions $\varphi,u$ on $V_x$ so that $\varphi$ is psh, $u$ is smooth, $dd^c\varphi=\gamma$, $dd^cu=(\pi^{-1})^\star\theta'$. Then the function $\varphi_\varepsilon=\varphi+u+\varepsilon F\circ\pi^{-1}$ is psh on $V_x$ and $dd^c\varphi_\varepsilon=\omega$. It follows by the definition \eqref{e:F} of $F$ that near $x$ we have $F\circ\pi^{-1}=-\log(g-\log|z_1|)+O(1)$, where $g$ is a smooth function. Thus 
$$\varphi_\varepsilon=\varphi-\log(g-\log|z_1|)+O(1),$$  
which shows that the Lelong numbers $\nu(\varphi_\varepsilon,x)=\nu(\varphi,x)$. Since $x\in A\setminus Y$ was arbitrary this implies that $\nu(\omega_\varepsilon,A)=\nu(\omega,A)$. 

\par Next, let $S\in H^0_{(2)}(X\setminus\Sigma,L^p,h_\varepsilon,f\Omega^n)$ be defined on $V_x$ by $S=se_\alpha^{\otimes p}$, where $e_\alpha$ is a local frame for $L$, and let $\nu=\nu(\omega,A)$. As $f\geq c$ we have 
$$\int_{V_x\setminus A}|s|^2e^{-2p\varphi_\varepsilon}\,d\lambda<\infty,$$ 
where $\lambda$ is the Lebesgue measure in coordinates. By the results of \cite{Siu74}, $dd^c\varphi_\varepsilon=\nu dd^c\log|z_1|+T$, where $T$ is a positive closed current, so $T=dd^cv$ for some psh function $v$ on $V_x$. It follows that the function $\varphi_\varepsilon-\nu\log|z_1|-v$ is pluriharmonic on $V_x$. Hence, by shrinking $V_x$ if necessary, we have 
$$\varphi_\varepsilon\leq\nu\log|z_1|+O(1), \text{ hence } \int_{V_x\setminus A}|s|^2|z_1|^{-2p\nu}\,d\lambda<\infty.$$ 
This implies that $s$ vanishes at least to order $p\nu$ along $A$.
\end{proof}

\begin{Remark} The proof of Proposition \ref{P:heps} shows in fact that the currents $\omega$ and $\gamma$ have the same Lelong numbers at each point of $\Sigma_{reg}^{n-1}$. However, the Lelong numbers of $\omega$ are bigger than those of $\gamma$ at other points of $\Sigma$. For instance, if $\Sigma$ is a finite set then $\widetilde X$ is simply the blow up of $X$ at each of the points of $\Sigma$. Then, in local coordinates $z$ near a point $x=0\in\Sigma$, we have $\pi_\star\theta'= a\,dd^c\log\|z\|$, for some $a>0$.
\end{Remark}

\subsection{Singular metrics on big line bundles}\label{SS:sm} Let $L$ be a big line bundle over the compact complex manifold $X$. Then $X$ is Moishezon and $L$ admits a singular metric $h$, smooth outside a proper analytic subset $\Sigma$ of $X$, and with strictly positive curvature current $\gamma=c_1(L,h)$ (see e.\,g.\ \cite[Lemma 2.3.6]{MM07}).

\subsubsection{Special metrics on Moishezon manifolds} Let $\Theta$ be the Poincar\'e metric on $X\setminus\Sigma$ and $h_\varepsilon$ be the small perturbation of the metric $h$ on $L$ constructed in Section \ref{SS:metrics}. It is shown in \cite{MM04,MM08} (see also \cite[Chapter 6]{MM07}) that the Bergman kernel function $P_p$ of the space $H^0_{(2)}(X\setminus\Sigma,L^p,h_\varepsilon,\Theta^n)$ has a full asymptotic expansion locally uniformly on $X\setminus\Sigma$. This clearly implies \eqref{e:mainhyp}, so we have the following:

\begin{Theorem} The conclusions of Theorems \ref{T:mt}, \ref{T:SZ} and \ref{T:rs} hold for the spaces $H^0_{(2)}(X\setminus\Sigma,L^p,h_\varepsilon,\Theta^n)$ and for $\omega=c_1(L,h_\varepsilon)$. 
\end{Theorem}

\par Note that in this case $X$ is not assumed to be K\"ahler.

\smallskip

\par Let us give here a simple alternate proof of the convergence $\frac{1}{p}\,\gamma_p\to\omega$, taking advantage of the fact that $X$ is compact. Let $Z_1=A_1\cup\ldots\cup A_l$, where $A_j$ are the irreducible components of $\Sigma$ of dimension $n-1$, and let $\eta$ be a Gauduchon form on $X$, i.e. a smooth positive (1,1) form with $dd^c(\eta^{n-1})=0$ \cite{Gau77}. Then by Lemma \ref{L:gammap}, 
$$\int_Xdd^c\log P_p\wedge\eta^{n-1}=0, \text{ so } \int_X\frac{1}{p}\,\gamma_p\wedge\eta^{n-1}=\int_X\omega\wedge\eta^{n-1}.$$ 
By Proposition \ref{P:heps}, the generic Lelong number of $\gamma_p/p$ on $A_j$ is at least $\nu_j=\nu(\omega,A_j)$. Siu's decomposition theorem (\cite{Siu74}, see also \cite{D93}) implies that 
$$\frac{1}{p}\,\gamma_p=R_p+\sum_{j=1}^l\nu_j[A_j]\,,\;\omega=R+\sum_{j=1}^l\nu_j[A_j],$$
where $[A_j]$ denotes the current of integration on $A_j$, $R_p,R$ are positive closed currents on $X$, and $R$ does not charge $Z_1$ (i.e. the trace measure of $R$ is 0 on $Z_1$). We have 
$$\int_XR_p\wedge\eta^{n-1}=\int_XR\wedge\eta^{n-1},$$ 
so the sequence of currents $\{R_p\}$ has uniformly bounded mass on $X$. It suffices to show that any limit point $T$ of this sequence is equal to $R$. By \eqref{e:mainhyp}, $T=R$ on $X\setminus\Sigma$. Hence by the support theorem $T=R$ on $X\setminus Z_1$, as $\Sigma=Z_1\cup Z_2$ and $\dim Z_2\leq n-2$. Since $T\geq0$ and $R$ does not charge $Z_1$ it follows that  $T\geq R$. 
But $T$ and $R$ have the same mass, so $T=R$. $\qed$

\subsubsection{Multiplier ideal sheaves} We recall first the notion of multiplier ideal sheaf. Let $\varphi\in L^1(X,\mathbb R)$. The {\em Nadel multiplier ideal sheaf} $\cI(\varphi)$ is the ideal subsheaf of germs of holomorphic functions $f\in\cO_{X,x}$ such that $|f|^2e^{-2\varphi}$ is integrable with respect to the Lebesgue measure in local coordinates near $x$.

\par If $h'$ is a  smooth Hermitian metric on $L$ then $h=h'e^{-2\varphi}$ for some function $\varphi\in L^1(X,\mathbb R)$. The Nadel multiplier ideal sheaf of $h$ is defined by $\cI(h)=\cI(\varphi)$; the definition does not depend on the choice of $h'$. The space of global sections in the sheaf $L\otimes\cI(h)$ is given by
\begin{equation}\label{l2:mult}
H^0(X,L\otimes\cI(h))=\Big\{
s\in H^0(X,L)\,:\,\int_X\big\lvert s\big\rvert^2_h\,\Omega^n=\int_X\big\lvert s\big\rvert^2_{h'}\,e^{-2\varphi}\,\Omega^n<\infty
\Big\}\,,
\end{equation}
where $\Omega$ is a fixed smooth positive $(1,1)$ form on $X$. We have 
$$H^0(X,L^p\otimes\cI(h_p))=H^0_{(2)}(X\setminus\Sigma,L^p,h,\Omega^n),$$ 
where $h_p$ is the metric induced by $h$ on $L^p$. If $\{S^p_j\}$ is an orthonormal basis of $H^0(X,L^p\otimes\cI(h_p))$ we define the Fubini-Study currents $\gamma_p$ on $X$ as in \eqref{e:gammap}.

\begin{Theorem} Let $L$ be a big line bundle over a compact K\"ahler manifold $X$ and $h$ be a singular Hermitian metric on $L$, smooth outside a proper analytic subset $\Sigma$ of $X$, and with strictly positive curvature current $\gamma=c_1(L,h)$. The conclusions of Theorems \ref{T:mt}, \ref{T:SZ} and \ref{T:rs} hold for the spaces $H^0(X,L^p\otimes\cI(h_p))$ and for $\gamma$.
\end{Theorem}

\begin{proof}
Conditions (A)\,-(C) are obviously verified in the present situation. Moreover, \eqref{e:mainhyp} follows from Theorem \ref{T:Bkf}. It can also be seen as a consequence of the full asymptotic expansion of the Bergman kernel function proved in \cite{HsM11}. Therefore, Theorem \ref{T:mt} implies the desired conclusion.
\end{proof}

\par Note that $X$ is in fact a projective manifold, since it is Moishezon and K\"ahler (see e.\,g.\ \cite[Th.\,2.2.26]{MM07}).

\subsection{Zariski-open manifolds with bounded negative Ricci curvature}\label{SS:qp}
Let $(M,J,\omega)$ be a K\"ahler manifold, let $g^{TM}$ be the Riemannian metric associated to $\omega$ by $g^{TM}(Ju,Jv)=  g^{TM}(u,v)$ for all $u,v\in T_xM$, $x\in M$. Let $\ric$ be the Ricci curvature of $g^{TM}$. The Ricci form $\ric_\omega$ is defined as the $(1,1)$-form associated to $\ric$ by
$$\ric_\omega(u,v)=\ric(Ju,v)\,,\quad\text{for any $u,v\in T_{x}M$, $x\in M$}.$$
The volume form $\omega^n$ induces a metric $h^{K^*_M}$ on $K^*_M$, whose dual metric on $K_M$ is denoted by $h^{K_M}$.
Since the metric $g^{TM}$ is K\"ahler, we have (see e.\,g.\ \cite[Prob.\,1.7]{MM07})
$$\ric_\omega=iR^{K_{M}^{*}}=-iR^{K_{M}}\,.$$
We denote by $H^0_{(2)}(M,K_M^p)$ the space of $L^2$-pluricanonical sections with respect to the metric $h^{K^p_M}$ and the volume form $\omega^n$.

\smallskip

\par We consider in this section the following setting:

\medskip

\par (I) $X$ is a compact complex manifold of dimension $n$, $\Sigma$ is an analytic subvariety of $X$, $M:=X\setminus\Sigma$. 

\smallskip

\par (II) $M$ admits a complete K\"ahler metric $\omega$ such that $\ric_\omega\leq-\lambda\omega$, for some constant $\lambda>0$.

\medskip

\par Note that $K_M=K_X\mid_{_\Sigma}$. Moreover, condition (II) implies that the volume form $\omega^n$ is integrable over $X$; indeed, by Yau's Schwarz lemma \cite[Theorem 3]{Ya:78} it follows that $\omega^n\lesssim\Theta^n$, where $\Theta$ is the generalized Poincar\'e metric on $M$ (see e.\,g.\ \cite[Prop. 1.10]{Nad:90}) and $\Theta^n$ is integrable over $X$. We have the following:

\begin{Theorem}\label{T:qpc2}
\par Let $X,\,\Sigma,\,M,\,\omega$ be as in (I), (II), and assume that $\dim\Sigma\leq n-k$, $k\geq2$. Then the following hold:

\par (i) $H^0_{(2)}(M,K_M^p)\subset H^0(X,K_X^p)$.

\par (ii) The currents $(-\ric_\omega)^j,\,\gamma_p^j$, $1\leq j\leq k$, are well defined on $X$ for $p$ sufficiently large, where $\gamma_p$ are the Fubini-Study currents defined by \eqref{e:gammap} for $H^0_{(2)}(M,K_M^p)$.

\par (iii) $\frac{1}{p^j}\gamma^j_p\to(-\frac{1}{2\pi}\ric_\omega)^j$  weakly on $X$ as $p\to\infty$, for $1\leq j\leq k$.
\end{Theorem}

\begin{proof} We only have to check condition (B). Since $\codim\Sigma\geq2$, Lemma \ref{L:hext} implies that the metric $h^{K_M}$ extends to a positively curved (singular) metric on $K_X$ over $X$, which we denote by $h$. Moreover, 
$$-\ric_\omega=iR^{K_{M}}=2\pi c_1(K_M,h^{K_M})=2\pi c_1(K_X,h)\mid_{_M}$$ 
extends to a positive closed current on $X$. 

\par Condition \eqref{e:mainhyp} holds, as shown by Tian \cite[Theorem 4.1]{Ti90} (this follows also from the more general result in \cite[Th.\,6.1.1]{MM07}). Therefore, Theorem \ref{T:mt} implies the desired conclusion.
\end{proof}

\par Note that Tian \cite[\S5]{Ti90} considered the situation when $X,\,\Sigma,\,M$ verify assumptions (I), (II), $X$ is projective and $k=1$. In that case he shows that the sections of $H^0_{(2)}(M,K_M^p)$ extend meromorphically to $X$, with poles of order at most $p-1$ along $\Sigma$, and $-\ric_\omega$ extends to a positive closed current on $X$ \cite[Theorem C]{Ti90}.

\par This situation is more difficult, as the metric $h^{K_M}$ does not extend to a positively curved metric on $K_X$. Nevertheless, we shall now show how this case fits into our framework from Theorem \ref{T:mt}. In view of Theorem \ref{T:qpc2} (and its proof), we may assume without loss of generality that 

\medskip

\par (III) $\Sigma$ has pure dimension $n-1$. 

\medskip

\par For this purpose, consider the line bundle $L:=K_X\otimes\cO_{X}(\Sigma)$, where $\cO_{X}(\Sigma)$ is the line bundle associated to the divisor $\Sigma$. Let $\sigma$ be the canonical section of $\cO_{X}(\Sigma)$ (cf.\ \cite[p.\,71]{MM07}) and denote by $h_\sigma$ the metric induced by $\sigma$ on $\cO_{X}(\Sigma)$ (cf.\ \cite[Example\,2.3.4]{MM07}). Note also that $c_1(\cO_X(\Sigma),h_\sigma)=[\Sigma]$ by \cite[(2.3.8)]{MM07}. Consider the metric naturally defined by $h^{K_M}$,
\begin{equation}\label{e:hsigma}
h_{M,\sigma}:=h^{K_M}\otimes h_\sigma\;\,\text{on}\;\,L\mid_{_M}=K_M\otimes\cO_X(\Sigma)\mid_{_M}\cong K_M.
\end{equation}

\smallskip

\par We recall the following simple fact, whose proof is left to the reader.

\begin{Lemma}\label{L:isomorph} Let $X,\,\Sigma,\,M$ verify assumptions (I) and (III). Assume that $(E,h^E)$ is a singular Hermitian line bundle on $X$ and $p\geq1$. Then 
$$\mathcal{I}_\sigma:H^0(M,E\mid_{_M})\longrightarrow H^0(M,(E\otimes\cO_X(\Sigma)^p)\mid_{_M}),\;\mathcal{I}_\sigma(S)=S\otimes\sigma^{\otimes p},$$
is an isomorphism and we have $|S|^2_{h^E}=|\mathcal{I}_\sigma(S)|^2_{h^E\otimes h_\sigma^p}$ pointwise on $M$, where $h_\sigma^p$ is the metric induced by $h_\sigma$ on $\cO_X(\Sigma)^p$. 
\end{Lemma}

\begin{Lemma}\label{L:htext} Let $X,\,\Sigma,\,M,\,\omega$ verify assumptions (I)-(III). The metric $h_{M,\sigma}$ defined in \eqref{e:hsigma} extends uniquely to a positively curved metric $h$ on $L$ over $X$. The curvature current $c_1(L,h)$ is independent of the choice of $\sigma$ and we have $c_1(L,h)\mid_{_M}=-\frac{1}{2\pi}\ric_\omega$. 
\end{Lemma}

\begin{proof} By Lemma \ref{L:hext} it suffices to show that the metric $h_{M,\sigma}$ extends near each regular point $x\in\Sigma$. We follow at first the argument of Tian from \cite[Lemma 5.1]{Ti90} to estimate the volume of $\omega$ as in \cite[(5.3)]{Ti90}. Let $\mathbb{D}$ be the unit disc in $\mathbb{C}$. Then $x\in\Sigma$ has a coordinate neighborhood $U_x$ such that
\begin{equation*} \label{compl12,17}
U_x\cong\mathbb{D}^n,\;\,x=0,\;\,U_x\cap\Sigma\cong\{z=(z_1,\ldots,z_n):\,z_1=0\},\;\,U_x\cap M\cong\mathbb{D}^\star\times\mathbb{D}^{n-1}.
\end{equation*}
Consider the complete hyperbolic metric $g_x$ on $\mathbb{D}^\star\times\mathbb{D}^{n-1}$ given by the product of the Poincar\'e metrics on $\mathbb{D}^\star$ and $\mathbb{D}$. By (II) and Yau's Schwarz lemma \cite[Theorem 3]{Ya:78}, the volume of $\omega$ is dominated on $U_x\cap M$ by a constant multiple of the volume of $g_x$. On a smaller polydisc $\mathbb{D}_r^\star\times\mathbb{D}_r^{n-1}$, $r<1$, the volume of $g_x$ is $\sim(|z_1|\log|z_1|)^{-2}$. It follows that 
$$\det[g_{jk}]\leq C(|z_1|\log|z_1|)^{-2}\;\,{\rm on}\;\,\mathbb{D}_r^\star\times\mathbb{D}_r^{n-1},\;\;{\rm where}\;\,\omega=i\sum_{j,k=1}^ng_{jk}dz_j\wedge d\overline z_k,$$
for some constant $C>0$. 

\par We may assume that there exists a psh weight $\varphi$ of the metric $h_{M,\sigma}$ on $U_x\cap M\cong\mathbb{D}^\star\times\mathbb{D}^{n-1}$. By above,
$$e^{2\varphi}=|z_1|^2\det[g_{jk}]\leq C(\log|z_1|)^{-2}\;\,{\rm on}\;\,\mathbb{D}_r^\star\times\mathbb{D}_r^{n-1},$$
which implies that $\varphi(z)\to-\infty$ as $z\to\Sigma$, so $\varphi$ is upper bounded near $x$. Hence $\varphi$ extends to a psh function on $U_x$, and $h_{M,\sigma}$ extends uniquely to a positively curved metric $h$ on $L$. Moreover, 
$$c_1(L,h)\mid_{_M}=c_1(K_M,h^{K_M})+c_1(\cO_X(\Sigma)\mid_{_M},h_\sigma)=-\frac{1}{2\pi}\ric_\omega.$$ 

\par Since $X$ is compact, any section $\sigma'$ of $\cO_X(\Sigma)$ that vanishes on $\Sigma$ is a constant multiple of $\sigma$, hence the metric $h_{\sigma'}$ is a constant multiple of $h_\sigma$. This shows that $c_1(L,h)$ is independent of the choice of $\sigma$.
\end{proof}

\begin{Theorem}\label{T:qpc1} Let $X,\,\Sigma,\,M,\,\omega$ verify assumptions (I)-(III). Let $h^{K_M}$ be the metric  induced by $\omega$ on $K_M$ and $H^0_{(2)}(M,K_M^p)$ be the space of $L^2$-pluricanonical sections with respect to the metric $h^{K^p_M}$ and the volume form $\omega^n$.
Then we have:

\par (i) The Fubini-Study currents $\gamma_p$ of $H^0_{(2)}(M,K_M^p)$ extend naturally as closed currents of order 0 on $X$ defined locally by formula \eqref{e:gammap}, and $\frac{1}{p}\,\gamma_p+[\Sigma]\geq0$ on $X$. 

\par (ii) $\frac{1}{p}\,\gamma_p+[\Sigma]$ converge weakly on $X$ to a positive closed current $T$ so that $T\mid_{_M}=-\frac{1}{2\pi}\ric_\omega$ and $T=c_1(L,h)$ for a singular Hermitian metric $h$ on $L=K_X\otimes\cO_{X}(\Sigma)$.
\end{Theorem}

\begin{proof} By \cite[Prop.\,1.11]{Nad:90} (see also \cite[Lemma 5.1]{Ti90}) the sections in $H^0_{(2)}(M,K_M^p)$ extend to meromorphic sections of $K^p_X$ over $X$, with poles in $\Sigma$ of order at most $p-1$. This yields $(i)$. 

\par Let $h_{M,\sigma}$ be the metric defined in \eqref{e:hsigma} on $L\mid_{M}$, and $h$ be its extension to $L$ provided in Lemma \ref{L:htext}, so $c_1(L,h)\mid_{_M}=-\frac{1}{2\pi}\ric_\omega$. It follows from Lemma \ref{L:isomorph} and \cite[Prop.\,1.11]{Nad:90} that $\mathcal I_\sigma(H^0_{(2)}(M,K_M^p))=H^0_{(2)}(M,L^p,h,\omega^n)\subset H^0(X,L^p)$. So $X,\,\Sigma,\,(L,h)$ and the volume form $\omega^n$ verify assumptions (A), (B), (C'), (D) (see Remark \ref{R:mtalt}). 

\par Lemmas \ref{L:isomorph} and \ref{L:htext} imply that $\mathcal I_\sigma$ maps an orthonormal basis of $H^0_{(2)}(M,K_M^p)$ onto an orthonormal basis of $H^0_{(2)}(M,L^p,h,\omega^n)$ and that the Bergman kernel functions $P_p$ defined by \eqref{e:Bergfcn} for these spaces are equal. Condition \eqref{e:mainhyp} holds, as shown by \cite[\S4]{Ti90} or \cite[Th.\,6.1.1]{MM07}. By Theorem \ref{T:mt} and Remark \ref{R:mtalt} we have $\frac{1}{p}\,\gamma'_p\to c_1(L,h)$ weakly on $X$, where $\gamma'_p$ are the Fubini-Study currents defined by \eqref{e:gammap} for $H^0_{(2)}(M,L^p,h,\omega^n)$. 

\par Observe that Lemmas \ref{L:isomorph} and \ref{L:htext} imply $\gamma'_p=\gamma_p+p[\Sigma]$ on $X$. This completes the proof.
\end{proof}

\begin{Remark}
Note that assumptions (I)\,-(III) are verified if $X$ is a compact projective manifold, $\Sigma$ is an effective divisor of $X$, and $L=K_X\otimes\cO(\Sigma)$ is ample, due to a result by R. Kobayashi \cite{Kob84} about the existence of K\"ahler-Einstein metrics on $X\setminus\Sigma$. Conversely, let $X,\,\Sigma,\,M,\,\omega$ verify assumptions (I)-(III) as in Theorem \ref{T:qpc1}. By the proof of \cite[Prop.\,1.12]{Nad:90} we see that the following properties hold:

\par (a) There exists $p_0$ such that $H^0(X,L^{p_0})$ separates the points of $M$ and gives local holomorphic coordinates on $M$, 

\par (b) $M$ is biholomorphic to a quasiprojective manifold; in fact the meromorphic Kodaira map 
$\Phi_{p_0}:X\dashrightarrow\mathbb{P}^N$ defined by $H^0(X,L^{p_0})$ induces a birational morphism to a normal projective variety $Y$ such that $\Phi_{p_0}(M)$ is Zariski open in $Y$ and $\Phi_{p_0}:M\longrightarrow\Phi_{p_0}(M)$ is biholomorphic,

\par (c) $L$ is big and $X$ is Moishezon.

\noindent Note that $L$ is not necessarily ample in the case of toroidal compactifications considered in the Section \ref{SS:aq}.
\end{Remark}

\subsection{Arithmetic quotients}\label{SS:aq}

\par Let $D$ be a bounded symmetric domain in $\mathbb{C}^n$ and let $\Gamma$ be a neat arithmetic group acting properly discontinuously on $D$ (see \cite[p. 253]{Mum77}). Then $U:=D/\Gamma$ is a smooth quasi-projective variety, called an arithmetic variety. By \cite{AMRT:10}, $U$ admits a smooth toroidal compactification $X$. In particular, $\Sigma:=X\setminus U$ is a divisor with normal crossings. The Bergman metric $\omega^{\mathcal B}_{D}$ on $D$ descends to a complete K\"ahler metric $\omega:=\omega^{\mathcal B}_{U}$ on $U$. Moreover, $\omega$ is K\"ahler-Einstein with $\ric_{\omega}=-\omega$. We denote by $h^{K_U}$ the Hermitian metric induced by $\omega$ on $K_U$. We wish to study the spaces $H^0_{(2)}(U,K_U^p)$ of $L^2$-pluricanonical sections with respect to the metric $h^{K^p_U}$ and the volume form $\omega^n$.

\par As in Section \ref{SS:qp}, consider the line bundle $L:=K_X\otimes\cO_{X}(\Sigma)$ and the metric $h_{U,\sigma}$ on $L\mid_{_U}$ defined in \eqref{e:hsigma}. By Lemma \ref{L:htext} $h_{U,\sigma}$ extends uniquely to a positively curved metric $h$ on $L$ and $c_1(L,h)\mid_{U}=\frac{\omega}{2\pi}$. Clearly, Theorem \ref{T:qpc1} holds in this setting: 

\begin{Theorem}\label{T:arit} Let $X$ be a smooth toroidal compactification of an arithmetic quotient $U=D/\Gamma$ and set $\Sigma=X\setminus U$, $L=K_X\otimes\cO_{X}(\Sigma)$. Let $\omega$ be the induced Bergman metric on $U$ and let $h^{K_U}$ be the metric  induced by $\omega$ on $K_U$.
Then we have:

\par (i) The metric $h^{K_U}$ defines a singular metric $h$ on $L$ such that $c_1(L,h)$ is a positive current on $X$ which extends $\frac{\omega}{2\pi}$. 

\par (ii) $H^0_{(2)}(U,L^p,h,\omega^n)\subset H^0(X,L^p)$ for all $p\geq1$, so the currents $\gamma_p$ given by \eqref{e:gammap} for $H^0_{(2)}(U,L^p,h,\omega^n)$ extend naturally to positive closed currents on $X$.

\par (iii) $\frac{1}{p}\,\gamma_p\to c_1(L,h)$ and $\frac{1}{p}\,[\sigma_p=0]\to c_1(L,h)$ in the weak sense of currents on $X$, for $\lambda_\infty$-a.e. sequence $\{\sigma_p\}_{p\geq1}\in{\mathcal S}_\infty$, where ${\mathcal S}_\infty,\lambda_\infty$ are as in Theorem \ref{T:rs}.
\end{Theorem}

\par By Lemma \ref{L:isomorph}, $H^0_{(2)}(U,K_U^p)\cong H^0_{(2)}(U,L^p,h,\omega^n)$. Let us describe this space in more detail. By \cite[Prop.\,3.3,\,3.4(b)]{Mum77},
\begin{equation*}
\begin{split}
H^0(X,L^p)\cong\big\{&\text{modular forms with respect to the $p$\,-th power} \\ &\text{of the canonical automorphy
factor}\big\}\,,
\end{split}
\end{equation*}
so $H^0_{(2)}(U,K_U^p)\subset H^0(X,L^p)$ are modular forms. 
The space 
$$H^0(X,L^p\otimes\cO_X(\Sigma)^{-1})=H^0(X,K_X^p\otimes\cO_X(\Sigma)^{p-1})$$ 
of modular forms vanishing on the boundary is called the space of cusp forms.

\par We will need the following definition from Mumford \cite[p.\,242]{Mum77}. Let $\mathbb{D}$ be the unit disc in $\mathbb{C}$. Every $x\in\Sigma$ has a coordinate neighborhood $V_x\cong\mathbb{D}^n$ such that for some $1\leq l\leq n$,
\begin{equation} \label{compl12,18}
V_x\cong\mathbb{D}^n,\;\,x=0,\;\,V_x\cap\Sigma\cong\{z=(z_1,\ldots,z_n):\,z_1z_2\ldots z_l=0\}\,.\end{equation}

\begin{Definition}\label{D:good} A smooth Hermitian metric $h$ on $L\mid_{_U}$ is said to be \emph{good} on $X$ if for all $x\in\Sigma$ and all holomorphic frames $e$ of $L$ in a neighborhood $V_x\cong\mathbb{D}^n$ of $x$ as in \eqref{compl12,18} we have 

\par (i) $|e|^2_h,|e|^{-2}_h\leq C(\sum_{k=1}^l\log|z_k|)^{2\alpha}$, for some $C>0$, $\alpha\geq1$\,, 

\par (ii) The forms $\eta=\partial\log|e|^2_h$ and $d\eta$ have Poincar\'e growth on $V$\,.
\end{Definition}

\par Examples of Hermitian line bundles with good metrics are provided by the following class of line bundles over arithmetic quotients considered by Mumford in \cite[p.\,256]{Mum77}. Namely, if $D$ is a bounded symmetric domain, then $D\cong K\backslash G$, where $G$ is a semi-simple adjoint group and $K$ a maximal compact subgroup.  Let $E_0$ be a $G$-equivariant holomorphic line bundle over $D$. Let $U=D/\Gamma$ be an arithmetic quotient and $X$ be a smooth toroidal compactification of $U$. Then $\Gamma$ acts on $E_0$ and $E_U=E_0/\Gamma$ is a holomorphic line bundle on $U$. Moreover, $E_0$ carries a $G$-invariant Hermitian metric $h_0$ which induces a Hermitian metric $h_U$ on $E_U$. By \cite[Main Theorem 3.1]{Mum77}, $E_U$ admits a unique extension to a holomorphic line bundle $\overline{E}$ over $X$ such that the metric $h_U$ on $\overline{E}\mid_U=E_U$ is good on $X$.
 
\par Consider the $G$-invariant line bundle $(E_0,h_0)=(K_D,h^{K_D})$ on $D$, where $h^{K_D}$ is induced by $\omega^{\mathcal B}_{D}$. Note that the Bergman metric $\omega^{\mathcal B}_{D}$ is $G$-invariant and so is $h^{K_D}$. Then $(E_U,h_U)=(K_U,h^{K_U})$, where $h^{K_U}$ is induced by $\omega^{\mathcal B}_{U}$. By \cite[Main Thr. 3.1, Prop. 3.4]{Mum77} the extension $\overline{K}_U$ of $K_U$ satisfies the following condition: for any $x\in\Sigma$ and any open coordinate neighborhood $V\cong\mathbb{D}^n$ of $x$ as in \eqref{compl12,18}, a holomorphic frame of $\overline{K}_U\mid_V$ is of the form $e=(z_1z_2\ldots z_l)^{-1}dz_1\wedge\ldots\wedge dz_n$. This shows that $\overline{K}_U\cong K_X\otimes\cO_X(\Sigma)=:L$ and the metric $h_{U,\sigma}$ induced by $h^{K_U}$ (see \eqref{e:hsigma}) on $L\mid_{_U}\cong K_U$ is good on $X$. Hence we obtain by condition \emph{(i)} of Definition \ref{D:good} that 
\begin{equation}\label{e:cangood}
\omega^n\gtrsim \prod_{j=1}^l |z_j|^{-2}\Big(\sum_{k=1}^l\log|z_k|\Big)^{-2\alpha}\Omega^n\quad\text{on $V\setminus\Sigma$}\,,
\end{equation}
where $\alpha\geq1$ and $\Omega$ is a positive (1,1) form on $X$, and  
\begin{equation}\label{e:cangood1}
|e|^2_{h_{U,\sigma}}\lesssim\Big(\sum_{k=1}^l\log|z_k|\Big)^{2\alpha}\quad\text{on $V\setminus\Sigma$}\,.
\end{equation}

\begin{Lemma}\label{l2=cusp}
Let $U=D/\Gamma$ be an arithmetic quotient and let $X$ be a smooth toroidal compactification of $U$. Then $H^0_{(2)}(U,K_U^p)\cong H^0(X,K_X^p\otimes\cO_X(\Sigma)^{p-1})$, i.\,e.\ the space of $L^2$-pluricanonical sections is the space of cusp forms.
\end{Lemma}

\begin{proof} By \cite[Prop.\,1.11]{Nad:90} we have $H^0_{(2)}(U,K_U^p)\subset H^0(X,K_X^p\otimes\cO_X(\Sigma)^{p-1})$. If $S\in H^0(X,K_X^p\otimes\cO_X(\Sigma)^{p-1})$, then $S=fe^{\otimes p}$, in a neighborhood $V_x$ of $x\in\Sigma$ as in \eqref{compl12,18}, where $f\in\cO(V_x)$ vanishes on $\Sigma$ and $e$ is a frame of $L$ over $V_x$\,. Estimate \eqref{e:cangood1} together with the fact that $\omega^n$ is integrable over $X$ \cite[Prop. 1.10]{Nad:90} imply that 
$$\int_{V_x\setminus\Sigma}|S|_{h_{U,\sigma}}^2\,\omega^n=\int_{V_x\setminus\Sigma}|f|^2|e^{\otimes p}|_{h_{U,\sigma}}^2\,\omega^n\lesssim\int_{V_x\setminus\Sigma}|f|^2\big(\sum_{k=1}^l\log|z_k|\big)^{2p\alpha}\,\omega^n<\infty,$$ 
thus $S\in H^0_{(2)}(U,L^p,h,\omega^n)\cong H^0_{(2)}(U,K_U^p)$.
\end{proof}

\par Theorem \ref{T:arit} shows that the zero-divisors of random cusp forms $\{\sigma_p\}$ (where $\sigma_p$ is a $p$\,-pluricanonical section) are equidistributed with respect to the extension of the Bergman metric on a smooth toroidal compactification of an arithmetic quotient. The equidistribution on the arithmetic quotient $D/\Gamma$ itself was shown in \cite{DMS}. In this framework the equidistribution of zeros is a variant of the Quantum Unique Ergodicity conjecture of Rudnick-Sarnak \cite{RuSa:94}, cf.\ Rudnick \cite{Rud05}, Holowinsky and Soundararajan \cite{HolSou:10}, Marshall \cite{Mars11}.

\par The case of arithmetic quotients of dimension $1$ is particularly interesting.

\begin{Corollary}
Let $\Gamma\subset SL_2(\mathbb Z)$ be a subgroup of finite index acting freely and properly discontinuously on the hyperbolic plane $\mathbb H$ via linear fractional transformations. Set $U=\mathbb H/\Gamma$ and let $\omega$ be the induced Poincar\'e metric on $U$. Let $X$ be a compact Riemann surface such that $U\subset X$ and $X\setminus U=\Sigma$ is a finite set. Let $L=K_X\otimes\cO_{X}(\Sigma)$. Then the space $S_{2p}(\Gamma)$ of cusp forms of weight $2p$ of $\Gamma$ 
is isomorphic to $H^0_{(2)}(U,K^{p}_{U})$ and assertions (i)-(iii) of Theorem \ref{T:arit} hold for the Fubini-Study currents $\gamma_p$ defined by $S_{2p}(\Gamma)$ and for the zero-sets of random sequences of cusp forms. 
\end{Corollary}

\medskip

\par We can extend the results of this section for the class of invariant line bundles considered by Mumford \cite[p.\,256]{Mum77}.

\begin{Theorem} Let $D,\,U,\,X,\,(E_0,h_0)$ be as above and assume that $iR^{(E_0,h_0)}\geq \varepsilon\omega^{\mathcal B}_{D}$ on $D$, for some $\varepsilon>0$. Let $(E_U,h_U)$ be the induced line bundle on $U$ and $E$ be its unique extension to $X$ so that the metric $h_U$ on $E\mid_U$ is good on $X$. Then $h_U$ extends to a singular Hermitian metric $h$ on $E$ such that $c_1(E,h)$ is a positive current on $X$ which extends $c_1(E_U,h_U)$, and the conclusions of Theorems \ref{T:mt}, \ref{T:SZ} and \ref{T:rs} hold for the spaces $H^0_{(2)}(U,E^p,h,\omega^n)$ and for $c_1(E,h)$. 
\end{Theorem}

\begin{proof}  Let $x\in\Sigma$ and $V$ be a coordinate neighborhood of $x$ as in \eqref{compl12,18} on which there exists a holomorphic frame $e$ of $E$. Then the local weights $\varphi=-\log|e|_{h_U}$ verify
$$-\log\left|\sum_{k=1}^l\log|z_k|\right|-\frac{\log C}{2\alpha}\leq\frac{\varphi}{\alpha}\leq\log\left|\sum_{k=1}^l\log|z_k|\right|+\frac{\log C}{2\alpha}\;\;{\rm on}\;V\setminus\Sigma.$$
Since the metric $h_U$ is positively curved, the function $\varphi$ is psh on $U\cap V=V\setminus\Sigma$. Hence $\varphi$ is psh on $V$, in view of the previous upper bound and of Lemma \ref{L:pshext} hereafter. Thus $c_1(E,h)\geq0$ and condition (B) is fulfilled.

\par To prove that (C) holds, we write $\omega^n=f\Omega^n$ for some fixed positive (1,1) form $\Omega$ on $X$. Let $x\in\Sigma^{n-1}_{reg}$ and local coordinates $z_1,\dots,z_n$ be chosen so that $x=0$, $\Sigma=\{z_1=0\}$. Estimate \eqref{e:cangood} implies that $f\gtrsim|z_1|^{-2}(\log|z_1|)^{-2\alpha}$ near $x$, where $\alpha\geq1$. Hence $f\geq c_x>0$ $\Omega^n$-a.e. in a neighborhood $U_x$ of each $x\in(X\setminus\Sigma)\cup\Sigma^{n-1}_{reg}$.

\par Condition \eqref{e:mainhyp} holds due to \cite[Th.\,6.1.1]{MM07}, which applies since $iR^{(E,h)}\geq \varepsilon\omega^{\mathcal B}_{U}$ on $U$. By Theorem \ref{T:mt} we infer the conclusion.
\end{proof}

\begin{Lemma}\label{L:pshext} Let $V\subset\mathbb{C}^n$ be an open set and $\Sigma$ be a proper analytic subvariety of $V$. Suppose that $u$ is a psh function on $V\setminus\Sigma$ which verifies 
$$u(z)\leq C_{z_0}\log|\log dist(z,\Sigma)|$$
for $z\in V\setminus\Sigma$ near each point $z_0\in\Sigma_{reg}^{n-1}$, with a constant $C_{z_0}>0$. Then $u$ is locally upper bounded near each point of $\Sigma$ hence it extends to a psh function on $V$. 
\end{Lemma}

\begin{proof} It suffices to show that $u$ is locally upper bounded near each point $z_0\in\Sigma_{reg}^{n-1}$. We may assume that $z_0=0$, $\Sigma=\{z_1=0\}\subset V=\mathbb{D}^n$ and that 
$u(z_1,z')\leq C\log|\log|z_1||$ for $z\in V$ with $0<|z_1|<e^{-1}$, where $C>0$ is a constant. The function $u(\cdot,z')$ is subharmonic on $\mathbb{D}\setminus\{0\}$, so $r\to \max_{|z_1|=r}u(z_1,z')$ is a convex function of $\log r$ for $r>0$. The above upper bound on $u$ implies that this function is also increasing, so $u$ is upper bounded in a neighborhood of $z_0=0$.
\end{proof}

\subsection{Remark on adjoint bundles}\label{SS:rmad} 
Let us consider the following setting:

\smallskip

(A$^\prime$) $X$ is a complex manifold of dimension $n$ (not necessarily compact), $\Sigma$ is a compact analytic subvariety of $X$, and $\Omega$ is a smooth positive $(1,1)$ form on $X\setminus\Sigma$.

\smallskip

(B$^\prime$) $(L,h)$ is a holomorphic line bundle on $X$ with a singular (semi)positive hermitian metric $h$ which is continuous on $X\setminus\Sigma$.

\smallskip

\par Consider the space $H^0_{(2)}(X\setminus\Sigma,L^p\otimes K_X)$ of $L^2$-holomorphic sections of $L^p\otimes K_X\mid_{_{X\setminus\Sigma}}$ relative to the metric $h_p$ induced by $h$ and the volume form  $\Omega^n$ on $X\setminus\Sigma$, endowed with the inner product
$$(S,S^{\,\prime})_p=\int_{X\setminus\Sigma}\langle S,S^{\,\prime}\rangle_{h_p,\Omega}\,\Omega^n\,,\;\;S,S^{\,\prime}\in H^0_{(2)}(X\setminus\Sigma,L^p\otimes K_X).$$
The interesting point is that the space $H^0_{(2)}(X\setminus\Sigma,L^p\otimes K_X)$ does not depend on the choice of the form $\Omega$. Indeed, for any $(n,0)$-form $S$ with values in $L^p$, and any metrics $\Omega$, $\Omega_1$ on $X\setminus\Sigma$, we have pointwise $\vert S\vert_{h_p,\Omega}^2\Omega^n=\vert S\vert_{h_p,\Omega_1}^2\Omega_1^n$. Therefore, we can take $\Omega$ to be a smooth positive (1,1) form on $X$. Then Skoda's lemma \cite[Lemma\,2.3.22]{MM07} shows that sections in $H^0_{(2)}(X\setminus\Sigma,L^p\otimes K_X)$ extend holomorphically to $X$, thus $H^0_{(2)}(X\setminus\Sigma,L^p\otimes K_X)\subset H^0(X,L^p\otimes K_X)$. Using an orthonormal basis of the space $H^0_{(2)}(X\setminus\Sigma,L^p\otimes K_X)$ we define the Bergman kernel function $P_{p}$ and the Fubini-Study currents $\gamma_{p}$ as in \eqref{e:Bergfcn}\,-\eqref{e:gammap}.

Proceeding as above for the proof of Theorem \ref{T:mt} we obtain the following. Under conditions (A$^\prime$)-(B$^\prime$) and assuming \eqref{e:mainhyp} we have $\frac{1}{p}\,\gamma_p\to\gamma$ weakly on $X$. If, in addition, $\dim\Sigma\leq n-k$ for some $2\leq k\leq n$, then the currents $\gamma^k$ and $\gamma_p^k$ are well defined on $X$, respectively on each relatively compact neighborhood of $\Sigma$, for all $p$ sufficiently large. Moreover, $\frac{1}{p^k}\,\gamma_p^k\to\gamma^k$ weakly on $X$.

\subsection{$1$-convex manifolds} \label{SS:1conv}
A complex manifold $X$ is called \emph{$1$-convex} if there exists a smooth exhaustion function $\psi:X\to\mathbb{R}$ which is strictly psh outside a compact set of $X$. This is equivalent to the following condition (see e.\,g.\ \cite{AG:62}): There exists a Stein space $Y$, a proper holomorphic surjective map $\rho : X\to Y$ satisfying $\rho_\star\mathcal{O}_X = \mathcal{O}_Y$ , and a finite set $A\subset Y$ such that the induced map $X\setminus\rho^{-1}(A) \to Y\setminus A$ is biholomorphic. The Stein space $Y$ is called the Remmert reduction of $X$ and $\Sigma:=\rho^{-1}(A)$ the exceptional set of $X$.

\par Consider a strictly psh smooth exhaustion function $\varphi_Y$ of $Y$, such that $\varphi_Y\geq0$ and $\{\varphi_Y=0\}=A$. Then $\varphi=\varphi_Y\circ\rho$ is a smooth psh exhaustion function of $X$, such that $\varphi\geq0$, $\{\varphi=0\}=\Sigma$ and $\varphi$ is strictly psh on $X\setminus\Sigma$.

\par We consider in the sequel a holomorphic line bundle $(L,h)$ on $X$ with singular metric $h$, which is smooth outside the exceptional set $\Sigma$ and has strictly positive curvature current in a neighborhood $U$ of $\Sigma$. By using a modification $\widetilde{X}$ of $X$ we construct as in Section \ref{SS:metrics} the Poincar\'e metric $\Theta$ on $X\setminus\Sigma$ and also the metric $h_{\varepsilon}$ on $L\mid_{_{X\setminus\Sigma}}$. We may suppose that $\Theta$ is complete on $X\setminus\Sigma$ (the metric $\widetilde{\Omega}$ on $\widetilde{X}$ may be taken to be complete, by setting $\widetilde{\Omega}=\Psi e^{\eta}$, where $\Psi$ is an arbitrary metric on $\widetilde{X}$ and $\eta$ is a fast increasing function at infinity).

\par Let us consider a convex increasing function $\chi:\mathbb{R}\to\mathbb{R}$ and endow $L$ with the Hermitian metric $h_{\varepsilon}e^{-\chi(\varphi)}$. Consider the $L^2$ inner product on the space $\Omega^{0,*}_0(X\setminus\Sigma,L^p)$ of sections with compact support, induced by the metrics $h_{\varepsilon}e^{-\chi(\varphi)}$ on $L$ and $\Theta$ on $X\setminus\Sigma$. Set
\begin{gather*}
L^2_{0,*}(X\setminus\Sigma,L^p):=L^2_{0,*}(X\setminus\Sigma,L^p,h_{\varepsilon}e^{-\chi(\varphi)},\Theta^n)\,,\\
H^0_{(2)}(X\setminus\Sigma,L^p):=L^2_{0,0}(X\setminus\Sigma,L^p)\cap\ker\overline\partial\,.
\end{gather*}
We denote by $\db^{\,*}_{\chi}$ and $\square_{p,\chi}$ the adjoint of $\db$ with respect to this 
$L^2$ inner product and the corresponding Kodaira Laplace operator. 

\par Let us denote by $\mathcal{T}=[i(\Theta),\partial\Theta]$ the Hermitian torsion of the Poincar\'e metric $\Theta$. Set $\widetilde{L^p}=L^p\otimes K^*_X$. There exists a natural isometry
\begin{equation*}
\begin{split}
&\Psi=\thicksim\,:\Lambda^{0,q}(T^*X)\otimes L^p
\longrightarrow\Lambda^{n,q}(T^*X)\otimes \widetilde{L^p},\\
&\Psi \, s=\widetilde s=(w^1\wedge\ldots\wedge w^n\wedge s)\otimes
(w_1\wedge\ldots\wedge w_n),
\end{split}
\end{equation*}
where $\{w_j\}^n_{j=1}$ is a local orthonormal frame of $T^{(1,0)}X$ and $\{w^j\}^n_{j=1}$ is the dual frame. The Bochner-Kodaira-Nakano formula \cite[Cor.\,1.4.17]{MM07} shows that for any $s\in \Omega^{0,1}_0(X\setminus\Sigma,L^p)$ we have
\begin{equation} \label{e:bkn}
\begin{split}
\frac{3}{2}\left(\norm{\db s}^2+\norm{\db^{\,*}_{\chi}s}^2\right)
\geq &\,
\left( R^{L^p\otimes K^*_X}(w_j,\ov{w}_k)\ov{w}^k\wedge i_{\ov{w}_j}s,
s\right)\\
&-\frac{1}{2}\big(\norm{\mathcal{T}^*\wi{s}}^2
+\norm{\ov{\mathcal{T}}\wi{s}}^2 
+\norm{\ov{\mathcal{T}}^{\,*}\wi{s}}^2\big).
\end{split}
\end{equation}
Set $T=\frac12(\mathcal{T}\,\mathcal{T}^*+\ov{\mathcal{T}}^{\,*}\,\ov{\mathcal{T}}+\ov{\mathcal{T}}\,\ov{\mathcal{T}}^{\,*})$. 
Define the continuous function
\begin{equation}\label{e:tau}
\tau:X\setminus\Sigma\to\mathbb{R}\,,\quad
\tau(x)=\sup\big\{\langle T\alpha,\alpha\rangle/\langle\alpha,\alpha\rangle : \alpha\in\Lambda^{n,1}T^{*}_xX\setminus\{0\}\big\}\,.
\end{equation}
Then 
for any $x\in X\setminus\Sigma$, $p\in\mathbb{N}$ and $\alpha\in L^p_x\otimes\Lambda^{n,1}T^{*}_xX$
we have
\[
\langle T\alpha,\alpha\rangle\leq\tau(x)\langle\alpha,\alpha\rangle\,.
\]
Hence \eqref{e:bkn} gives for all $s\in \Omega^{0,1}_0(X\setminus\Sigma,L^p)$
\begin{equation} \label{e:bkn1}
\frac{3}{2}\left(\norm{\db s}^2+\norm{\db^{\,*}_{\chi}s}^2\right)
\geq \,
\left( R^{L^p\otimes K^*_X}(w_j,\ov{w}_k)\ov{w}^k\wedge i_{\ov{w}_j}s\,,
s\right)-\int_{X\setminus\Sigma}\tau(x)|s|^2\,.
\end{equation}

\begin{Lemma}\label{l:herm} 
There exists an increasing convex function $\chi:\mathbb{R}\to\mathbb{R}$ and constants $a,b>0$, such that:
\begin{subequations}
  \begin{equation}\label{e:es1}	
  c_1(L,h_{\varepsilon}e^{-\chi(\varphi)})\geq a\Theta\,,
  \end{equation}
  \begin{equation}\label{e:es2}	
	c_1(L,h_{\varepsilon}e^{-\chi(\varphi)})+iR^{K_X^*}-\tau\Theta\geq -b\Theta\,,
\end{equation}
	\end{subequations}
on $X\setminus\Sigma$.
\end{Lemma}  

\begin{proof}
We have 
\[
c_1(L,h_{\varepsilon}e^{-\chi(\varphi)})=c_1(L,h_{\varepsilon})+\frac{1}{2}\,dd^c\chi(\varphi)=c_1(L,h_{\varepsilon})+\frac{i}{2\pi}\,(\chi^\prime(\varphi)\partial\overline\partial\varphi+\chi^{\prime\prime}(\varphi)\partial\varphi\wedge\overline\partial\varphi).
\] 
Since $\varphi$ is psh, for any increasing convex function $\chi$ this is $\geq c_1(L,h^L)$, hence positive on $U$. Thus \eqref{e:es1} holds on $U$ by construction of $h_{\varepsilon}$. Moreover,
\cite[Lemma\,6.2.1]{MM07} shows that $iR^{K_X^*}$ and the torsion operators of $\Theta$, hence $\tau$, are bounded with respect to $\Theta$ on $U$. Thus \eqref{e:es2} also holds on $U$, thanks to \eqref{e:es1}.
  
\par Consider $c>0$ such that $\Sigma\subset X_c\Subset U$, where $X_c:=\{\varphi<c\}$. Note that $\varphi$ is strictly psh outside $X_c$. Thus we we can choose $\chi$ increasing fast enough such that
\eqref{e:es1}\,-\eqref{e:es2} are satisfied on $X\setminus\ov{X}_c$.
\end{proof}

\begin{Lemma} \label{l:herm1}
Let $\chi:\mathbb{R}\to\mathbb{R}$ be as in Lemma \ref{l:herm}. Then:
\\[2pt]
(i) There exist constants
$a_1,b_1>0$ such that for any $p\in\mathbb{N}$  we have
\begin{equation}\label{bk1.411}
\norm{\db s}^2+\norm{\db^{\,*}_{\chi}s}^2\geq (p\,a_1-a_1-b_1)\|s\|^2\,,\quad s\in\Dom(\db)\cap\Dom(\db^{\,*}_{\chi})\cap L^2_{0,1}(X\setminus\Sigma,L^p).
\end{equation}
(ii) The spectrum of $\square_{p,\chi}$ on $L^2_{0,0}(X\setminus\Sigma,L^p)$ satisfies
\begin{equation}\label{e:spgap}
\operatorname{Spec}(\square_{p,\chi})\subset \{0\}\cup\, (p\,a_1-a_1-b_1,+\infty).
\end{equation}
(iii) 
The Bergman kernel function $P_p$ of $H^0_{(2)}(X\setminus\Sigma,L^p)$ has a full asymptotic expansion  
on any compact set of $X\setminus\Sigma$.
\end{Lemma}  

\begin{proof} (i) Since $\Omega^{0,1}_0(X\setminus\Sigma,L^p)$
is dense in $\Dom(\db)\cap\Dom(\db^{\,*}_{\chi})\cap L^2_{0,1}(X\setminus\Sigma,L^p)$ (Andreotti-Vesentini density lemma, see \cite[Lemma\,3.3.1]{MM07}) it suffices to prove \eqref{bk1.411} for $s\in\Omega^{0,1}_0(X\setminus\Sigma,L^p)$. But in this case, \eqref{bk1.411} follows immediately from \eqref{e:bkn1} and \eqref{e:es1}\,-\eqref{e:es2}. 
\\[2pt]
(ii) Once we have \eqref{bk1.411}, the assertion about the spectrum of $\square_{p,\chi}$ on $L^2_{0,0}(X\setminus\Sigma,L^p)$ follows as in the proof
of \cite[Th.\,6.1.1]{MM07}.
\\[2pt]
(iii) Since the Kodaira Laplacian $\square_{p,\chi}$ on $L^2_{0,0}(X\setminus\Sigma,L^p)$ has a
spectral gap, by the argument in \cite[\S 4.1.2]{MM07}, we can localize the problem,
and we obtain the result from \cite[Th.\,4.2.9]{MM07}, as in the proof of \cite[Th.\,4.2.1]{MM07}.
\end{proof}

\begin{Theorem}\label{T:1conv}
Let $X$ be a $1$-convex manifold and $(L,h)$ be a holomorphic line bundle on $X$ with singular metric $h$. Assume that $h$ is smooth outside the exceptional set $\Sigma$ and that it has strictly positive curvature current in a neighborhood of $\Sigma$. Let $\Theta$ be a complete Poincar\'e metric on $X$, $h_{\varepsilon}$ be constructed in Section \ref{SS:metrics}, and let $\chi$ be as in Lemma \ref{l:herm}. The conclusions of Theorems \ref{T:mt}, \ref{T:SZ} and \ref{T:rs} hold for the spaces $H^0_{(2)}(X\setminus\Sigma,L^p,h_{\varepsilon}e^{-\chi(\varphi)},\Theta^n)$ and for $\omega=c_1(L,h_{\varepsilon}e^{-\chi(\varphi)})$. 
\end{Theorem}

\begin{proof}
Conditions (A)\,-(C) are satisfied by construction and condition \eqref{e:mainhyp} follows from Lemma \ref{l:herm1}. 
\end{proof}

\subsection{Strongly pseudoconvex domains}\label{SS:spscd}
We give now a variant `with boundary' of the previous result. Let $M$ be a complex manifold and let $X\Subset M$ be a strongly pseudoconvex domain with smooth boundary. We consider a defining function $\varrho\in\mathcal{C}^\infty (M,\mathbb{R})$ of $X$, i.e., $X=\{x\in M:\varrho(x)<0\}$ and $d\varrho\neq0$ on $\partial X$. Since $X$ is strongly pseudoconvex, the Levi form of $\varrho$ is positive definite on the complex tangent space to $\partial X$. It is well-known that one can modify the defining function $\varrho$ such that in a neighborhood of $\partial X$, $\varrho$ is strictly psh and $d\varrho\neq0$. Thus, for $c\geq0$ small enough, $X_c=\{x\in M:\varrho(x)<c\}$ is strongly pseudoconvex.  

\par Let $\eta_c:(-\infty,c)\to\mathbb{R}$ be a convex increasing function such that $\eta_c(t)\to\infty$, as $t\to c$. Then $\eta_c\circ\varrho$ is an exhaustion function for $X_c$, which is strictly psh outside a compact set of $X_c$. Therefore $X_c$ is a $1$-convex manifold. 

\par Let $\Sigma$ be the exceptional set of $X_c$ (it is the same exceptional set as for $X$) and let $\varphi:X_c\to\mathbb{R}$ be a smooth psh exhaustion function of $X_c$, such that $\varphi\geq0$,
$\{\varphi=0\}=\Sigma$ and $\varphi$ is strictly psh on $X_c\setminus\Sigma$.

\par Let $(L,h)$ be a holomorphic line bundle on $M$ with singular metric $h$ which is smooth outside the exceptional set $\Sigma$ and which has strictly positive curvature current in a neighborhood $U$ of $\Sigma$. By using a modification $\widetilde{M}$ of $M$ we construct as in Section \ref{SS:metrics} the Poincar\'e metric $\Theta$ on $M\setminus\Sigma$ and also the metric $h_{\varepsilon}$ on $L\mid_{_{M\setminus\Sigma}}$.

\par Let $A>0$. On the space $\Omega^{0,*}_0(\overline{X}\setminus\Sigma,L^p)$ of sections with compact support in $\overline{X}\setminus\Sigma$ we introduce the $L^2$ inner product with respect to the metrics $\Theta$ and $h_{\varepsilon}e^{-A\varphi}$ and set 
\begin{gather*}
L^2_{0,*}(X\setminus\Sigma,L^p):=L^2_{0,*}(X\setminus\Sigma,L^p,h_{\varepsilon}e^{-A\varphi},\Theta^n)\,,\\
H^0_{(2)}(X\setminus\Sigma,L^p):=L^2_{0,0}(X\setminus\Sigma,L^p)\cap\ker\overline\partial\,.
\end{gather*}
We consider the $L^2$ $\db$-Neumann problem on $\overline{X}\setminus\Sigma$ and show that the $\db$-Neumann Laplacian on $L^2_{0,1}(X\setminus\Sigma,L^p)$ has a spectral gap. Here we work with $\db$-Neumann boundary conditions at the end $\partial X$ of $\overline{X}\setminus\Sigma$
and with a complete metric at the end corresponding to $\Sigma$. This kind of analysis was already used in \cite{MD02} in connection to the compactification of hyperconcave manifolds. 

\par We denote by $\overline\partial^{\,*}=\overline\partial^{L^p,*}$ the Hilbert space adjoint of the maximal extension of $\overline\partial$ on $L^2_{0,1}(X\setminus\Sigma,L^p)$. Integration by parts \cite[Prop. 1.3.1--2]{FK:72} yields 
\begin{eqnarray*}
B^{0,1}(\overline{X}\setminus\Sigma,L^p)&:=&\{s\in\Omega^{0,1}_0(\overline{X}\setminus\Sigma,L^p):
*\partial\varrho\wedge*s=0\, \, \, \text{on}\,\,\partial X\}\\
&=&\Dom(\overline\partial^{\,*})\cap\Omega^{0,1}_0(\overline{X}\setminus\Sigma,L^p)\,.
\end{eqnarray*}
The space $B^{0,1}(\overline{X}\setminus\Sigma,L^p)$ is dense in $\Dom(\db)\cap\Dom(\overline\partial^{\,*})$ with respect to the graph norm $s\mapsto (\|s\|^2+\|\db s\|^2+\|\overline\partial^{\,*}s\|^2)^{1/2}$ (cf.\ \cite[Lemma\,2.2]{MD02}).

\par Let us consider a defining function $\varrho$ of $X$ such that $|d\varrho|=1$ on $\partial X$. We denote by $\mathscr{L}_{\varrho}$ the Levi form of $\varrho$ (cf.\ \cite[Def.\,1.4.20]{MM07}). The Bochner-Kodaira-Nakano formula with boundary term \cite[Cor.\,1.4.22]{MM07} shows that for any $s\in B^{0,1}(\overline{X}\setminus\Sigma,L^p)$ we have
\begin{equation*}
\begin{split}
\frac{3}{2}\left(\norm{\db s}^2+\norm{\db^{\,*}s}^2\right)
\geq &\,
\left(R^{L^p\otimes K^*_X}(w_j,\ov{w}_k)\ov{w}^k\wedge i_{\ov{w}_j}s,
s\right)\\
&+\int_{\partial X}\mathscr{L}_{\varrho}(s,s)\,dv_{\partial X}
-\frac{1}{2}\big(\norm{\mathcal{T}^{\,*}\wi{s}}^2
+\norm{\ov{\mathcal{T}}\wi{s}}^2 
+\norm{\ov{\mathcal{T}}^{\,*}\wi{s}}^2\big).
\end{split}
\end{equation*}
Since $X$ is strongly pseudoconvex the boundary integral is non-negative. 
Therefore we obtain for all $s\in B^{0,1}(\overline{X}\setminus\Sigma,L^p)$ the estimate 
\begin{equation*}
\frac{3}{2}\left(\norm{\db s}^2+\norm{\db^{\,*}_{\chi}s}^2\right)
\geq \,
\left( R^{L^p\otimes K^*_X}(w_j,\ov{w}_k)\ov{w}^k\wedge i_{\ov{w}_j}s\,,
s\right)-\int_{\ov{X}\setminus\Sigma}\tau(x)|s|^2\,,
\end{equation*}
where $\tau$ is defined on $X_c\setminus\Sigma$ as in \eqref{e:tau}.
Making use of the compactness of $\ov{X}$ we obtain:  

\begin{Lemma}\label{l6.1} 
There exist constants $A,a,b>0$ such that $c_1(L,h_{\varepsilon}e^{-A\varphi})$ is a strictly positive $(1,1)$ current on a neighborhood of $\overline{X}$ and 
\begin{subequations}
  \begin{equation}\label{e:es11}	
  c_1(L,h_{\varepsilon}e^{-A\varphi})\geq a\Theta\,,
  \end{equation}
  \begin{equation}\label{e:es21}	
	c_1(L,h_{\varepsilon}e^{-A\varphi})+iR^{K_X^*}-\tau\Theta\geq -b\Theta\,,
\end{equation}
	\end{subequations}
on $\ov{X}\setminus\Sigma$.
\end{Lemma}  

\par Let us now fix $A>0$ as in Lemma \ref{l6.1}. Using \eqref{e:bkn1}, \eqref{e:es11} and \eqref{e:es21},
we deduce immediately the estimate \eqref{bk1.411} for any $s\in B^{0,1}(\overline{X}\setminus\Sigma,L^p)$ and, by density, for any $s\in\Dom(\db)\cap\Dom(\db^{\,*})\cap L^2_{0,1}(X\setminus\Sigma,L^p)$.
This shows that $\square_{p}$ acting on $L^2_{0,0}(X\setminus\Sigma,L^p)$ has a spectral gap as in \eqref{e:spgap}.
Therefore, the Bergman kernel function $P_p$ of $H^0_{(2)}(X\setminus\Sigma,L^p)$ has a full asymptotic expansion  
on any compact set of $X\setminus\Sigma$.

\par The preceding discussion leads to the following.

\begin{Theorem}\label{T:spsc}
Let $X$ be a strongly pseudoconvex domain with smooth boundary in a complex manifold $M$.
Let $(L,h)$ be a holomorphic line bundle on $M$ with singular metric $h$ which is smooth outside the exceptional set $\Sigma$ and which has strictly positive curvature current in a neighborhood $U$ of $\Sigma$. The conclusions of Theorems \ref{T:mt}, \ref{T:SZ} and \ref{T:rs} hold for the spaces $H^0_{(2)}(X\setminus\Sigma,L^p,h_{\varepsilon}e^{-A\varphi},\Theta^n)$ and for $\omega=c_1(L,h_\varepsilon e^{-A\varphi})$. 
\end{Theorem}

\begin{Remark}
In the same vein, we can obtain a variant of Theorem \ref{T:spsc} for Nadel multiplier sheaves. 
Assume that $X\Subset M$ is a strongly pseudoconvex domain as above. Let $(L,h)$ be a holomorphic line bundle on $M$ with singular metric $h$ which is smooth outside the exceptional set $\Sigma$. Assume for simplicity that the curvature current of $h$ is strictly positive in a whole neighborhood of $\overline{X}$. The conclusions of Theorems \ref{T:mt}, \ref{T:SZ} and \ref{T:rs} hold for the spaces $H^0(X,L^p\otimes\cI(h_p))$ {\rm(}defined as in \eqref{l2:mult}{\rm)} and for $\gamma=c_1(L,h)$. 
\end{Remark}

\end{document}